\documentclass[review]{elsarticle}

\usepackage{lineno}
\modulolinenumbers[5]

\journal{Journal of Mathematical Analysis and Applications}

\RequirePackage{amssymb,amsfonts,amsmath}
\usepackage{amsthm}
\RequirePackage{bm,bbm}
\RequirePackage{epsfig,graphicx}
\RequirePackage{comment}
\usepackage{lscape}
\usepackage{mathtools}
\usepackage{float}
\usepackage{subfigure}
\usepackage[export]{adjustbox}
\usepackage{scalerel}

\DeclarePairedDelimiter\floor{\lfloor}{\rfloor}
\newtheorem{theorem}{Theorem}

\newtheorem{assumption}{Assumption}

\def \td  {{\underline{t}}}

\def \hX  {{\widehat{X}}}
\def \hY  {{\widehat{Y}}}
\def \bX  {{\overline{X}}}

\def \EE  {{\mathbb{E}}}
\def \PP  {{\mathbb{P}}}
\def \QQ  {{\mathbb{Q}}}
\def \RR  {{\mathbb{R}}}
\def \Rr  {{\mathrm{R}}}

\def \D   {{\rm d}}

\def \e   {{\rm e}}  

\begin{document}

\begin{frontmatter}

\title{Multilevel Monte Carlo Method for Ergodic SDEs without Contractivity}
\author[mysecondaryaddress]{Wei Fang\corref{mycorrespondingauthor}}
\cortext[mycorrespondingauthor]{Corresponding author}
\ead{wei.fang@maths.ox.ac.uk}
\author[mysecondaryaddress]{Michael B. Giles}
\ead{mike.giles@maths.ox.ac.uk}


\address[mysecondaryaddress]{Mathematical Institute, University of Oxford}

\begin{abstract}
This paper proposes a new multilevel Monte Carlo (MLMC) method for the ergodic SDEs which do not satisfy the contractivity condition. By introducing the change of measure technique, we simulate the path with contractivity and add the Radon-Nikodym derivative to the estimator. We can show the strong error of the path is uniformly bounded with respect to $T.$ Moreover, the variance of the new level estimators increase linearly in $T,$ which is a great reduction compared with the exponential increase in standard MLMC. Then the total computational cost is reduced to $O(\varepsilon^{-2}|\log \varepsilon|^{2})$ from $O(\varepsilon^{-3}|\log \varepsilon|)$ of the standard Monte Carlo method. Numerical experiments support our analysis.
\end{abstract}


\end{frontmatter}

\linenumbers

\section{Introduction}
In this paper we consider an $m$-dimensional stochastic 
differential equation (SDE) driven by an $m$-dimensional 
Brownian motion:
\begin{equation}
\D X_t = f(X_t)\,\D t + \,\D W_t,
\label{SDE}
\end{equation}
which has a Lipschitz drift $f: \RR^m\!\rightarrow\!\RR^m$ satisfying the dissipativity condition: for some $\alpha,\beta>0,$
\begin{equation}
\langle x,f(x)\rangle \leq -\alpha \|x\|^2+\beta.
\label{eq:dissipative}  
\end{equation}
Theorem 6.1 in \cite{MT93b} shows that this class of SDEs is ergodic and solutions converge exponentially to some invariant measure $\pi$. Evaluating the expectation of some function $\varphi(x)$ with respect to that invariant measure $\pi$ is of great interest in mathematical biology, physics and Bayesian inference in statistics:
$$\pi(\varphi)\coloneqq\int\varphi(x)\,\mathrm{d}\pi(x)=\lim_{t\rightarrow \infty}\mathbb{E}\left[\varphi(X_t) \right],\ \ \varphi\in L^1(\pi). $$
Different approaches to computing the expectation include numerical solution of the Fokker-Planck equation, see \cite{So94} and the references therein, and estimation of the time average of the ergodic numerical solutions, see \cite{MT93b,Ha03,Le07,MSH02,MST10,MT13,
RT96,Ta90}.   

One simple way is to use one of the existing numerical methods for finite time SDEs to simulate the SDE for a sufficiently long time $T,$ see \cite{Part1,HMS02,HJ15,kp92,MT07} and references therein. The  exponential convergence to the invariant measure \cite{MT93b} is given by
\begin{equation}
\left|\EE\left[\varphi(X_T)-\pi(\varphi)\right]\right|\leq \mu^*\,\e^{-\lambda^*T}
\label{ConvergenceRateInvariant}
\end{equation}
for some constant $\mu^*,\lambda^*>0,$ and bounding this truncation error by $\varepsilon$ requires 
\begin{equation}
\label{eq:log T}
T\,\geq\, \frac{1}{\lambda^*} \log(\varepsilon^{-1})+\frac{\log \mu^*}{\lambda^*},
\end{equation}
which means the computational cost of each path using uniform time step $h=O(\varepsilon)$ becomes $O(\varepsilon^{-1}|\log \varepsilon|)$ for numerical schemes with first order weak convergence.
Theorem 1 in \cite{Part2} shows that under the dissipativity condition (\ref{eq:dissipative}) the $p$-th moments of the numerical solution are bounded uniformly with respect to $T,$ so the variance of the estimator is bounded by a constant $V_0$ which does not depend on $T.$ Therefore, the computational cost to achieve $\varepsilon^2$ mean square error (MSE) is $O(\varepsilon^{-3}|\log \varepsilon|).$

The multilevel Monte Carlo (MLMC) method, introduced by Giles \cite{giles08,giles15}, can be applied to reduce the computational cost. If the SDEs further satisfy the contractivity condition: for all $x,y\in\RR^m$,
\begin{equation}
\langle x\!-\!y,f(x)\!-\!f(y)\rangle\,\leq\, -\lambda\,\|x\!-\!y\|^2,
\label{eq:contraction_Lipschitz}
\end{equation}
for some $\lambda>0,$ Theorem 3 in \cite{Part2} has proved first order strong convergence and that the strong error is uniformly bounded with respect to $T.$ Hence, the variance of the multilevel correction $V_\ell$ on each level $\ell$ is bounded by $Ch_\ell^2$ with $C>0$ not depending on $T.$ The MLMC computational cost to achieve $\varepsilon^2$ MSE becomes $O(\varepsilon^{-2}|\log \varepsilon|),$ where the additional $O(|\log\varepsilon|)$ comes from the length of simulation time $T.$ In \cite{Part2}, by simulating different time intervals $T_\ell$ across different levels $\ell$, we further reduce the computational cost to $O(\varepsilon^{-2}).$ 

However, a larger class of SDEs satisfying the dissipativity condition (\ref{eq:dissipative}) does not satisfy the contractivity condition and instead only satisfies the one-sided Lipschitz condition:
\begin{equation}
\langle x\!-\!y,f(x)\!-\!f(y)\rangle\,\leq\, \lambda\,\|x\!-\!y\|^2,
\label{eq:one_sided_Lipschitz}
\end{equation}
for some $\lambda>0.$ The major benefit of the contractivity is that two solutions to the SDE starting from different initial data but driven by the same Brownian motion, will converge exponentially, which means the discretization error from previous time steps will decay exponentially, and then we can prove a uniform bound for the strong error. Without the contractivity, the strong error may increase exponentially with respect to $T.$ Then multilevel correction variances $V_\ell$ also increase exponentially, which, as shown in Theorem \ref{thm: MLMC}, increases the total computational cost to $ O(\varepsilon^{-2-\frac{\kappa}{2\lambda^*}} |\log\varepsilon|),$ where $\kappa$ is the Lyapunov exponent of the system. For some SDEs with a chaotic property, the Lyapunov exponent $\kappa$ can be sufficiently large such that $\frac{\kappa}{2\lambda^*}\geq 1$ and MLMC loses its advantage over the standard Monte Carlo method.


In this paper, a change of measure technique is employed to deal with SDEs satisfying the one-sided Lipschitz condition (\ref{eq:one_sided_Lipschitz}). We provide the numerical analysis only for the case of a globally Lipschitz drift but this scheme works well for SDEs with non-globally Lipschitz drift such as the stochastic Lorenz equation which is only locally one-sided Lipschitz. 

The key feature of this class of SDEs, especially the chaotic SDEs, is that the behaviour of solutions is highly sensitive to initial conditions and the difference between the fine path and coarse path will increase exponentially. An intuitive way to avoid this kind of divergence is by adding a "spring" between the fine path and coarse path to draw them closer to each other. 

Mathematically, instead of simulating the fine path and coarse path of the original SDEs, that is, in their separated path spaces with different measures:
\begin{eqnarray*}
\label{Eq: SDE before change}
\mathbb{Q}^f:\ \ \ \D X_t^f &=& f(X_t^f)\,\D t + \D W_t^{\mathbb{Q}^f},\nonumber\\
\mathbb{Q}^c:\ \ \ \,\D X_t^c &=& f(X_t^c)\,\D t + \D W_t^{\mathbb{Q}^c}.
\end{eqnarray*}
We add a spring term with spring coefficient $S>\frac{\lambda}{2}$ for both fine path and coarse path for all $\ell>1,$ and simulate the fine path and coarse path in the same probability measure $\mathbb{P}$:
\begin{eqnarray}
\label{Eq: SDE for change}
\D Y_t^f &=& S(Y_t^c-Y_t^f)\,\D t + f(Y_t^f)\,\D t + \D W_t^\PP,\nonumber\\
\D Y_t^c &=& S(Y_t^f-Y_t^c)\,\D t + f(Y_t^c)\,\D t + \D W_t^\PP.
\end{eqnarray}
The Girsanov theorem gives
\begin{equation}
\EE^{\mathbb{Q}^f}[X_t^f]-\EE^{\mathbb{Q}^c}[X_t^c]=
\EE^{\mathbb{P}}\left[Y_t^f\frac{\D \mathbb{Q}^f}{\D \mathbb{P}}-Y_t^c\frac{\D \mathbb{Q}^c}{\D \mathbb{P}}\right],
\end{equation}
where $\frac{\D \mathbb{Q}^f}{\D \mathbb{P}}$ and $\frac{\D \mathbb{Q}^c}{\D \mathbb{P}}$ are the corresponding Radon-Nikodym derivatives of the measure $\mathbb{Q}^f$ on the fine path space and measure $\mathbb{Q}^c$ on the coarse path space with respect to the $\mathbb{P}$ measure in which we are simulating both paths. In practice, we will derive the Radon-Nikodym derivative exactly for the numerical solution instead of numerically approximating the derivatives above. In the new MLMC scheme, essentially, the fine path $Y_t^f$ and coarse path $Y_t^c$ share the same driving Brownian motion $W_t$ in measure $\mathbb{P}.$ Correspondingly, the Brownian motions for the original SDEs, 
\begin{eqnarray*}
\D W_t^{\mathbb{Q}^f} &=& S(Y_t^c-Y_t^f)\,\D t+ \D W_t^\PP,\\
\D W_t^{\mathbb{Q}^c} &=& S(Y_t^f-Y_t^c)\,\D t + \D W_t^\PP,
\end{eqnarray*}
are slightly different in measure $\mathbb{P}$, which is different from the standard MLMC. 
The benefit of this change is that the difference between the new simulated SDEs satisfies
\begin{equation}
\label{adaptive S}
\D (Y_t^f- Y_t^c)= 2S(Y_t^c-Y_t^f)\,\D t + (f(Y_t^f)- f(Y_t^c))\,\D t,
\end{equation}
and provided $S>\frac{\lambda}{2}$, Ito's formula and the one-sided Lipschitz condition (\ref{eq:one_sided_Lipschitz}) give:
\begin{equation*}
\D\, \|Y_t^f- Y_t^c\|^2 \leq 2(\lambda-2S)\|Y_t^f- Y_t^c\|^2\, \D t,
\end{equation*}  
which will recover the contractivity between the fine and coarse paths. Note that the choice of the simple form of the spring term is motivated
by this intuitive explanation and makes it easy to prove that
contractivity is recovered.  It also works well in practice, but we do
not claim it is optimal and further research is required to investigate
and analyse possible improvements.  Due to the contractivity, we can
prove that the strong difference between the coarse and fine paths is
uniformly bounded with respect to T.  More importantly, we can show that, together with the Radon-Nikodym derivatives, the variance of the new MLMC correction estimator increases only linearly in $T,$ which is a great improvement compared with the exponential increase without the change of measure. 
The total computational cost can be reduced to $O(\varepsilon^{-2} |\log\varepsilon|^{2}),$ where the order is independent of the convergence rate $\lambda^*$ of the original SDE and the Lyapunov exponent $\kappa.$

Change of measure techniques have been used in previous research to reduce the variance of corrections in MLMC. Giles proposed to use the same Gaussian samples for the final step of both fine and coarse paths with a change of measure for the pricing of the digital option on page 38 in \cite{giles15}. To cope with the SDEs with path-dependent jumps, Xia \& Giles \cite{Xia} used a change of measure so that the acceptance probability of the jumps is the same for both fine and coarse paths. Kebaier \& Lelong \cite{KL17} optimize over a class of measures to optimally reduce the variance of MLMC corrections. Andersson \& Kohatsu-Higa \cite{AK17} change the sampling distribution to make the MLMC correction variance finite for unbiased simulation of SDEs using parametrix expansions. Stilger \& Poon \cite{SP14} apply it for MLMC calculation of an interest rate model and Gasparotto \cite{GA15} for deep out-of-money options to reduce the variance.  

The change of measure technique together with the Lamperti transform is also the core part in the exact simulation of SDEs, see \cite{Exact} and the references therein. Importance sampling (change of measure technique) has also been widely used in rare event simulations, see \cite{DE01,JS06} for a good introduction and review and the references therein. 

Lastly, the construction of good coupling between paths is also useful for theoretical results. Eberle et al \cite{EGZ17} proposed a new coupling method to estimate the theoretical convergence rate for Langevin dynamics. See \cite{CL89} and its references for further exploration.

The rest of the paper is organised as follows.  Section 2 introduces the new MLMC method with the change of measure. Section 3 states the main theorems, and the relevant numerical experiments are provided in section 4. Numerical results for SDEs with non-globally Lipschitz drift are given in section 5. The proofs of the main theorems are deferred to section 6, and finally, section 7 has some conclusions and discusses future extensions.

In this paper we consider the infinite time interval $[0,\infty)$ and let 
$(\Omega,\mathcal{F},\PP)$ be a probability space with 
normal filtration $(\mathcal{F}_t)_{t\in[0,\infty)}$ corresponding to
a $m$-dimensional standard Brownian motion 
$W_t^\PP=(W^{(1)},W^{(2)},\ldots,W^{(m)})_t.$
We denote the vector norm by 
$\|v\|\triangleq(|v_1|^2+|v_2|^2+\ldots+|v_m|^2)^{\frac{1}{2}}$, 
the inner product of vectors $v$ and $w$ by 
$\langle v,w \rangle\triangleq v_1w_1+v_2w_2+\ldots+v_mw_m$, 
for any $v,w\in\RR^m$ and the Frobenius matrix norm by 
$\|A\|\triangleq \sqrt{\sum_{i,j}A_{i,j}^2}$ 
for any $A\in\RR^{m\times d}.$

\section{New MLMC with change of measure}
In this paper, we use the standard Euler-Maruyama method to simulate the original SDE (\ref{SDE}) using uniform timestep $h>0$ under measure $\PP$:
\begin{equation}
\label{EM step}
t_{n+1}=t_{n}+h,\ \ \hX_{t_{n+1}} = \hX_{t_n}+ f(\hX_{t_n})h + \Delta W_n^{\PP}, 
\end{equation}
where $\Delta W_n^{\PP} \triangleq W_{t_{n+1}}^\PP-W_{t_n}^\PP$ for $n=0,1,...,N-1$ with $N=T/h$ and there is fixed initial data $t_0=0,\ \hX_{t_0}=x_0.$ We use the notation 
$
\td \triangleq \max\{t_n: t_n\!\leq\! t\},\
n_t\triangleq\max\{n: t_n\!\leq\! t\}
$
for the nearest time point before time $t$, and its index.
We define the piece-wise constant interpolant process $\bX_t=\hX_\td$ 
and also define the standard continuous interpolant \cite{kp92} as
\[
\hX_t=\hX_\td+f(\hX_\td) (t\!-\!\td) + (W_t\!-\!W_\td).
\]
Then, the standard Monte Carlo estimator for $\EE^{\PP}\left[ \varphi(X_T) \right]$ is the mean of the values $\varphi(\hX_T^L)$, from $N_L$ independent path simulations using $h=2^{-L}h_0$ for some suitable constant $h_0>0$ and positive integer $L.$
\begin{equation}
\label{Std estimator}
\widehat{\varphi}_{std} \coloneqq N_L^{-1} \sum_{n=1}^{N_L} \varphi(\hX_T^{L,(n)}).
\end{equation}
Next, we quickly review the standard MLMC scheme introduced in \cite{giles08,giles15}. Instead of directly estimating $\EE^{\PP}\left[\varphi(\hX_T^L)\right],$ we have the following telescoping sum in the same probability measure $\mathbb{P}$:
\[\EE^{\PP}\left[\varphi(\hX_T^L)\right] = \EE^{\PP}\left[\varphi(\hX_T^0)\right] + \sum_{\ell=1}^L \EE^{\PP}\left[\varphi(\hX_T^{f,\ell}) -\varphi(\hX_T^{c,\ell-1})\right], \]
where $\hX_T^{f,\ell}$ and $\hX_T^{c,\ell-1}$ share the same driving Brownian motion. Then, the standard MLMC estimator becomes
\begin{equation}
\label{MLMC estimator}
\widehat{\varphi}_{mlmc} \coloneqq N_0^{-1} \sum_{n=1}^{N_0} \varphi(\hX_T^{0,(n)}) + \sum_{\ell=1}^L N_\ell^{-1} \sum_{n=1}^{N_\ell}\left(  \varphi(\hX_T^{f,\ell,(n)})- \varphi(\hX_T^{c,\ell-1,(n)}) \right).
\end{equation}
Now we introduce the new MLMC scheme with change of measure using spring coefficient $S>0.$

For level 0, the numerical estimator is the same as the standard MLMC $\varphi(\hX_T^0)$.   

For level $\ell>1,$ we simulate the SDE with the additional spring terms using timestep $h =  2^{-\ell}\,h_0$ for the fine path and $2h$ for the coarse path.
\begin{itemize}
\item
At $t_0,$ we set $\hY_{t_0}^f = \hY_{t_0}^c = x_0$.
\item
At odd timesteps $t_{2n+1}=t_{2n}+h$ for $n\geq 0,$ we update both paths:
\begin{eqnarray*}
 \hY_{t_{2n+1}}^c &=&\hY_{t_{2n}}^c+ S(\hY_{t_{2n}}^f-\hY_{t_{2n}}^c)h + f(\hY_{t_{2n}}^c) h + \Delta W_{2n}^\PP, \\
\hY_{t_{2n+1}}^f &=&\hY_{t_{2n}}^f+ S(\hY_{t_{2n}}^c-\hY_{t_{2n}}^f)h + f(\hY_{t_{2n}}^f) h + \Delta W_{2n}^\PP.
\end{eqnarray*} 
\item
At even timesteps $t_{2n+2}=t_{2n+1}+h$ for $n\geq 0,$ we update the spring term and drift term of the fine path, but keep both the same for the coarse path:
\begin{eqnarray*}
\hY_{t_{2n+2}}^c &=&\hY_{t_{2n+1}}^c+\ \  S(\hY_{t_{2n}}^f-\hY_{t_{2n}}^c)h\ \ \ + f(\hY_{t_{2n}}^c) h\ \ \, +\Delta W_{2n+1}^\PP,\\
 \hY_{t_{2n+2}}^f &=&\hY_{t_{2n+1}}^f+ S(\hY_{t_{2n+1}}^c-\hY_{t_{2n+1}}^f)h + f(\hY_{t_{2n+1}}^f) h +\Delta W_{2n+1}^\PP.
\end{eqnarray*} 
\end{itemize}
Note that the coarse path updates can be combined to give
\[
\hY_{t_{2n+2}}^c\ =\ \hY_{t_{2n}}^c+  S(\hY_{t_{2n}}^f-\hY_{t_{2n}}^c)2h + f(\hY_{t_{2n}}^c) 2h +\Delta W_{2n}^\PP+\Delta W_{2n+1}^\PP.
\]

Next, we derive the exact Radon-Nikodym derivatives for both fine and coarse paths. To begin with, suppose we only apply the change of measure to the $n$th timestep. Under measure $\PP,$ we have
\[\hY_{t_{n+1}} = \hX_{t_n} + \widehat{S}h + f(\hX_{t_n})h + \Delta W_n^\PP\ \Rightarrow\ \hY_{t_{n+1}}\sim N^\PP(\hX_{t_n} + f(\hX_{t_n})h + \widehat{S}h, hI),\]
where  $I$ is the identity matrix, $N^\PP(\mu,\Sigma)$ is a Normal distribution under measure $\PP$ and $\widehat{S}$ is the spring term. Under a new measure $\widehat{\mathbb{Q}}_n$ with $\Delta W_n^{\widehat{\QQ}_n} = \widehat{S}h +  \Delta W_n^\PP,$ we get
\[\hY_{t_{n+1}} = \hY_{t_n} + f(\hY_{t_n})h + \Delta W_n^{\widehat{\QQ}_n}\ \Rightarrow\ \hY_{t_{n+1}}\sim N^{\widehat{\QQ}_n}(\hY_{t_n} + f(\hY_{t_n})h, hI).\]
Then the exact Radon-Nikodym derivative for this single step is
\[\frac{\D \widehat{\mathbb{Q}}_n}{\D\mathbb{P}} = \frac{\rho (\hY_{t_{n+1}}\,|\hY_{t_n} + f(\hY_{t_n})h, hI )}{\rho(\hY_{t_{n+1}}\,| \hY_{t_n} + f(\hY_{t_n})h + \widehat{S}h, hI )} \coloneqq \mathrm{R}(\,\hY_{t_{n+1}},\hY_{t_n}, \widehat{S},h),\]
where $\rho(x|\mu,\Sigma)$ is the probability density function of $N(\mu,\Sigma)$ and
\begin{eqnarray*}
\mathrm{R}(\,\hY_{t_{n+1}},\hY_{t_n}, \widehat{S},h) &=& \exp\left(-\left\langle \hY_{t_{n+1}}-\hY_{t_n}-f(\hY_{t_n})h,\widehat{S}\right\rangle+\|\widehat{S}\|^2h/2 \right)\\
&=& \exp \left(-\left\langle  \Delta W_n^\PP ,\ \widehat{S} \right\rangle - \|\widehat{S}\|^2h/2 \right).
\end{eqnarray*}
Now, suppose that we introduce such changes on each timestep of the
whole path, so under a new measure $\widehat{\QQ},$ we have $\Delta W_n^{\widehat{\QQ}} = \widehat{S}h +  \Delta W_n^\PP,$ for $n=0,1,..., N-1.$ Since $\Delta W_n^{\PP}$ and $\Delta W_n^{\widehat{\QQ}},$ $n=0,1,..., N-1,$ are sets of independent Brownian
increments under measure $\PP$ and $\widehat{\QQ}$ respectively, the exact Radon-Nikodym derivative becomes
\[\frac{\D \widehat{\mathbb{Q}}}{\D\mathbb{P}}  \coloneqq \prod_{n=0}^{N-1}\mathrm{R}(\,\hY_{t_{n+1}},\hY_{t_n}, \widehat{S},h).\]
Numerically we obtain two new measures $\widehat{\QQ}^f$ and $\widehat{\QQ}^c$ with $\Delta W_n^{\widehat{\QQ}^f} = \widehat{S}^f_nh +  \Delta W_n^\PP$ and $\Delta W_n^{\widehat{\QQ}^c} = \widehat{S}^c_nh +  \Delta W_n^\PP$ respectively for all steps on fine and coarse paths, where  $\widehat{S}^f_n$ and $\widehat{S}^c_n$ are the spring terms on $n$th step for fine and coarse paths. Then we can calculate the exact Radon-Nikodym derivatives step by step at the same time as updating the paths.
\begin{itemize}
\item At $t_0,$ we set $\Rr_{t_0}^f = \Rr_{t_0}^c = 1$.
\item At odd timesteps $t_{2n+1}=t_{2n}+h$ for $n\geq 0,$ we only update $\Rr^f$:
\begin{equation*}
\Rr_{t_{2n+1}}^f\ \ =\ \ \Rr_{t_{2n}}^f\  \Rr\left(\,\hY_{t_{2n+1}}^f,\ \hY_{t_{2n}}^f,\ S(\hY_{t_{2n}}^c-\hY_{t_{2n}}^f),\ h\right).
\end{equation*}
\item At even timesteps $t_{2n+2}=t_{2n+1}+h$ for $n\geq 0,$ we update both $\Rr^f$ and $\Rr^c$:
\end{itemize}
\begin{eqnarray*}
 \Rr_{t_{2n+2}}^f &\ =\  &  \Rr_{t_{2n+1}}^f  \Rr\left(\, \hY_{t_{2n+2}}^f ,\hY_{t_{2n+1}}^f,S(\hY_{t_{2n+1}}^c-\hY_{t_{2n+1}}^f),h  \right),\\
\Rr_{t_{2n+2}}^c &\  =\  &  \Rr_{t_{2n}}^c\ \ \,\Rr \left(\, \hY_{t_{2n+2}}^c ,\ \hY_{t_{2n}}^c,\ S(\hY_{t_{2n}}^f-\hY_{t_{2n}}^c),\ 2h\right).
\end{eqnarray*}
Then, after $N$ steps, we obtain the exact Radon-Nikodym derivatives for the whole path:
\begin{eqnarray}
\label{eq: RD}
\frac{\D \widehat{\mathbb{Q}}^f}{\D \mathbb{P}} &=&\Rr^f_T = \prod_{n=0}^{N-1} \ \ \, \Rr\left(\,\hY_{t_{n+1}}^f,\ \hY_{t_{n}}^f,\  S(\hY_{t_{n}}^c-\hY_{t_{n}}^f),\  h\right),\nonumber\\
\frac{\D \widehat{\mathbb{Q}}^c}{\D \mathbb{P}}  &=& \Rr^c_T = \prod_{n=0}^{N/2-1}  \ \Rr \left(\, \hY_{t_{2n+2}}^c ,\ \hY_{t_{2n}}^c,\ S(\hY_{t_{2n}}^f-\hY_{t_{2n}}^c),\ 2h\right).
\end{eqnarray}
Finally, the multilevel correction estimator becomes 
\begin{equation}
\label{eq: level estimator}
 \varphi(\hY_T^f)\,\Rr^f_T-\varphi(\hY_T^c)\,\Rr^c_T,
\end{equation}
and the identity we use in the new MLMC is 
\[
\EE^{\PP}\left[\varphi(\hX_T^L)\right] = \EE^{\PP}\left[\varphi(\hX_T^0)\right] + \sum_{\ell=1}^L \EE^{\PP}\left[\varphi(\hY_T^{f,\ell})\Rr^{f,\ell}_T -\varphi(\hY_T^{c,\ell-1}\Rr^{c,\ell}_T)\right]
\]
where $\Rr^{f,\ell}_T$ and $\Rr^{c,\ell}_T$ are the exact Radon-Nickodym derivatives for the fine and coarse paths on level $\ell.$ The new MLMC estimator becomes
\begin{eqnarray}
\label{new MLMC estimator}
\widehat{\varphi}_{new} &\coloneqq & N_0^{-1} \sum_{n=1}^{N_0} \varphi(\hX_T^{0,(n)})  \\
&&+ \sum_{\ell=1}^L N_\ell^{-1} \sum_{n=1}^{N_\ell}\left(  \varphi(\hX_T^{f,\ell,(n)})\Rr^{f,\ell,(n)}_T- \varphi(\hX_T^{c,\ell-1,(n)}) \Rr^{c,\ell,(n)}_T\right). \nonumber
\end{eqnarray}

In the following sections, we only work under measure $\PP,$ so we use $W_t$ to denote $W_t^\PP$ for simplicity.


 


\section{Theoretical Results}
In this section, we state the key results on the stability and strong error of the path after the change of measure, and then the variance of the estimator (\ref{eq: level estimator}) and the resulting MLMC complexity.
\begin{assumption}[Lipschitz and dissipativity]
\label{assp:linear_growth}
Assume $f$ is globally Lipschitz so that there is a constant $K>0$ such that
\begin{equation}
\| f(x)\!-\!f(y) \| \leq K\, \|x\!-\!y\| ,
\label{Lipschitz condition}
\end{equation}
for all $x,y\in \RR^m$.
Furthermore, there exist constants $\tilde{\alpha}, \tilde{\beta} > 0$ 
such that for all $x\in \RR^m$, $f$ satisfies the dissipativity condition:
\begin{equation}
\langle x,f(x)\rangle \leq -\tilde{\alpha} \|x\|^2+ \tilde{\beta}.
\label{eq:dissipativity}
\end{equation}
\end{assumption}
Note that a consequence of the Lipschitz condition is that 
\[\|f(x)\|\leq \|f(0)\|+K\|x\|\ \Rightarrow\ \|f(x)\|^2\leq 2\left(\|f(0)\|^2+K^2\|x\|^2\right).\]
This assumption ensures the existence and uniqueness of the strong solution to the SDEs \cite{MX07} and the convergence to the invariant distribution \cite{MSH02}. Note that the Lipschitz assumption is needed for simplicity of the proof but numerical experiments in section \ref{Non-Lipschitz section} show that the change of measure technique also works well for SDEs with non-globally Lipschitz drift. 
The following theorem, based on this assumption, shows that our numerical scheme with sufficiently small $h$ is stable and the moments of the numerical solution is uniformly bounded with respect to $T.$

\begin{theorem}[Stability]
\label{Theorem: stability of path}
If the original SDE satisfies Assumption  \ref{assp:linear_growth}, then using 
the new change-of-measure algorithm with $S\!>\!0$, 
there exist constants $C_{(1)}, C_{(2)}>0$ such that 
for any $T\!>\!0$ and $p\!\geq\! 1$, and
for all $0\!<\!h\!<\!C_{(1)}$
\[
\sup_{0\leq n \leq N} \EE\left[ \| \widehat{Y}^f_{t_n} \|^p \right]^{1/p} 
\leq C_{(2)}\, p^{1/2}, ~~
\sup_{0\leq n \leq N} \EE\left[ \| \widehat{Y}^c_{t_n} \|^p  \right]^{1/p}  
\leq C_{(2)}\, p^{1/2}.
\]
\end{theorem}
\begin{proof}
The proof is deferred to section \ref{proof 1}.
\end{proof}
It is important to note that the constants $C_{(1)}, C_{(2)}$ depend on the specifics of the original SDE and the value of $S,$ but not on $T,\,h$ or the moment power $p.$ 
This result is expected since the spring term is only a linear function of the numerical solution and the magnitude is small which does not destroy the dissipativity condition and allow us to obtain the uniform bounds.  
For the first-order strong convergence, we need the following assumption.
\begin{assumption}[One-sided Lipschitz properties]
\label{assp:enhanced_one_sided_Lipschitz}
There exists a constant $\lambda\!>\!0$ such that for all $x,y\in\RR$,
$f$ satisfies the one-sided Lipschitz condition:
\begin{equation}
\langle x - y,f(x) - f(y)\rangle \leq \lambda\,\|x - y\|^2,
\label{eq:one_sided_Lipschitz3}
\end{equation}
and $f$ is differentiable and $\nabla f(x)$ satisfies the Lipschitz condition $\|\nabla f(x) - \nabla f(y)\|\leq K\|x-y\|.$
\end{assumption}

Note that the Lipschitz condition (\ref{Lipschitz condition}) implies this one-sided Lipschitz condition (\ref{eq:one_sided_Lipschitz3}). However, the one-sided Lipschitz condition can give a sharper bound for the positive side, which means that $K$ can be much larger than $\lambda.$ The spring term in our algorithm is only needed when the inner product $\langle x - y,f(x) - f(y)\rangle$ is positive, to prevent the exponential divergence of the fine and coarse paths. See the adaptive spring for double-well potential energy SDE in section \ref{Non-Lipschitz section} where we choose $S$ to be a function of the current state to minimize the spring term and thereby reduce the size of the Radon-Nikodym derivative. The other consideration is that possibly we can extend this scheme to SDEs with locally one-sided Lipschitz drift, for example the stochastic Lorenz equation. Therefore, this condition helps us to obtain an accurate choice of spring term $S$ as shown in the following theorem.

\begin{theorem}[Difference between fine and coarse paths]
\label{thm: strong error}
If the original SDE satisfies Assumptions \ref{assp:linear_growth}  and  \ref{assp:enhanced_one_sided_Lipschitz}, then using 
the new change-of-measure algorithm with $S\!>\!\lambda/2$, 
there exist constants $C_{(1)},$ $C_{(2)}>0$ such that 
for any $T\!>\!0$ and $p\!\geq\! 1$, and 
for all $0\!<\!h\!<\!C_{(1)},$
\[
\sup_{0\leq n \leq N} \EE\left[ \| \widehat{Y}^f_{t_n} - \widehat{Y}^c_{t_n} \|^p \right]^{1/p} 
\leq C_{(2)}\, \min\left( p^{1/2}\, h^{1/2},\ p\, h\right).
\]
\end{theorem}

\begin{proof}
The proof is deferred to section \ref{proof 2}.
\end{proof}
The $L_p$ norm of the difference between the fine and coarse paths, as we expected, is uniformly bounded since we add enough spring term to recover the contractivity used in \cite{Part2}. With this result, we can bound the $p$th-moment of the Radon-Nikodym derivatives and then the MLMC estimator (\ref{eq: level estimator}).

\begin{theorem}[Radon-Nikodym moments]
\label{lemma 1}
If the original SDE satisfies Assumptions \ref{assp:linear_growth} and \ref{assp:enhanced_one_sided_Lipschitz}, then using 
the new change-of-measure algorithm with $S\!>\!\lambda/2$,
there exist constants $C_{(1)}, C_{(2)}>0$ such that, 
for any $T\!>\!0$ and $p\!\geq\! 1$, and 
for all $0\!<\!h < \min(C_{(1)},C_{(2)}/(Tp^2)),$
\begin{equation*}
\EE\left[ \left|\frac{\D \widehat{\mathbb{Q}}^c}{\D \mathbb{P}} \right|^p \right]
\leq 2, \ \ \ \ \EE\left[ \left|\frac{\D \widehat{\mathbb{Q}}^f}{\D \mathbb{P}} \right|^p \right] 
\leq 2.
\end{equation*}
\end{theorem}

\begin{proof}
The proof is deferred to section \ref{proof 4}.
\end{proof}

\begin{theorem}[MLMC moments]
\label{thm: level variance}
If the original SDE satisfies Assumptions \ref{assp:linear_growth} and \ref{assp:enhanced_one_sided_Lipschitz}, and
$\varphi: \RR^m \rightarrow \RR$ is globally Lipschitz, then using the 
new change-of-measure algorithm with $S\!>\!\lambda/2$, for any
$T\!>\!0$ and $p\!\geq 1$ there exists constants
$C_{(1)}, C_{(2)}, C_{(3)} >0$ such that for all 
$0< h< \min(C_{(1)},C_{(2)}/(Tp^2),$
\[
\EE\left[ \left|
\varphi(\widehat{Y}^f_T)\, \frac{\D \widehat{\QQ}^f}{\D \PP} -
\varphi(\widehat{Y}^c_T)\, \frac{\D \widehat{\QQ}^c}{\D \PP} 
\right|^p \right]^{1/p} 
\leq C_{(3)}\, p^{2} \, \sqrt{T}\, h.
\]
\end{theorem}

\vspace{-1em}

\begin{proof}
The proof is deferred to section \ref{proof 3}.
\end{proof}

Note that this theorem implies that the variance of the estimator (\ref{eq: level estimator}) is bounded by $C_2Th^2$ which increases linearly in $T.$

We now have everything we require to determine the MLMC complexity.

\begin{theorem}[MLMC for invariant measure]If $\varphi$ satisfies the Lipschitz condition and the SDE satisfies Assumption \ref{assp:linear_growth} and \ref{assp:enhanced_one_sided_Lipschitz} with convergence rate $\lambda^*$ and constant $\mu^*$ in \eqref{ConvergenceRateInvariant}, and Lyapunov exponent $\kappa$, then by choosing suitable values for $L,$ $S,$ $h_0$ and $N_\ell$ for each level $\ell,$ there exist constants $c_1,\,c_2,\, c_3>0$ such that the estimator $\widehat{\varphi}$ has a mean square error (MSE) with bound
\[\EE\left[(\widehat{\varphi}-\pi(\varphi))^2\right]\leq \varepsilon^2,\]   
with $0<\varepsilon<1$ and an expected computational cost $\mathrm{C}_{std}$ for the standard Monte Carlo estimator $\widehat{\varphi}_{std}$ \eqref{Std estimator} with bound
\[
\mathrm{C}_{std} \leq\ c_1 \ \varepsilon^{-3}|\log \varepsilon|,
\]
and an expected computational cost $\mathrm{C}_{mlmc}$ for the standard MLMC  estimator $\widehat{\varphi}_{mlmc}$ \eqref{MLMC estimator} with bound
\[
\mathrm{C}_{mlmc} \leq\ c_2\  \varepsilon^{-2-\frac{\kappa}{2\lambda^*}} |\log\varepsilon|,
\]
provided $\kappa/\lambda^*<2,$ and $\mathrm{C}_{com}$ for the new MLMC estimator with change of measure $\widehat{\varphi}_{new}$ \eqref{new MLMC estimator} with bound
\[
\mathrm{C}_{com} \leq\ c_3\  \varepsilon^{-2}|\log \varepsilon|^{2}.
\]
\label{thm: MLMC}
\end{theorem}
\vspace{-3em}

\begin{proof}
By Jensen's inequality, the MSE can be decomposed into three parts:
\begin{eqnarray*}
\EE\left[(\widehat{\varphi}-\pi(\varphi))^2\right]&=& \mathbb{V}\left[\widehat{\varphi}\right]+ \left|\EE\left[\widehat{\varphi}\right]-\pi(\varphi)\right|^2\\
&\leq& \mathbb{V}\left[\widehat{\varphi}\right]+ 2\left| \EE\left[\widehat{\varphi}\right]\!-\EE\left[\varphi(X_{T})\right]\right|^2\!\!+ 2\left|\EE\left[\varphi(X_{T})\right]-\pi(\varphi)\right|^2,
\end{eqnarray*}
which enables us to achieve the MSE bound by bounding each part 
by $\varepsilon^2/3$. Similar to \eqref{eq:log T}, we bound the third part by setting 
\begin{equation}
T\,=\, \frac{1}{\lambda^*} \log (\varepsilon^{-1})+\frac{\log \sqrt{6}\mu^*}{\lambda^*},
\label{eq: log T 2}
\end{equation}
to bound the truncation error. The first order weak convergence requires $h_L = O(\varepsilon)$ and
$L\geq\left\lceil \gamma \log_2 (\varepsilon^{-1})+\zeta \right\rceil $ for some $\gamma,\zeta>0.$

For the standard Monte Carlo method using $h_L,$ the computational cost for each path is $O(\varepsilon^{-1}|\log \varepsilon|)$ and the bound on variance requires $O(\varepsilon^{-2})$ samples, which gives a total computational cost
\[
\mathrm{C}_{std} \leq\ c_1 \ \varepsilon^{-3}|\log \varepsilon|,
\]
for some constant $c_1>0.$

The analysis for the two MLMC schemes is similar to the MLMC theorem in \cite{giles15} and shows the optimal computational cost is bounded by
\[
3\varepsilon^{-2}\, \left(\sum_{\ell=0}^{L} \sqrt{V_\ell\, \mathrm{C}_\ell} \right)^2\ + \ \sum_{\ell=0}^L \mathrm{C}_\ell\ ,
\]
where $\mathrm{C}_\ell$ and $V_\ell$ are the cost and variance for each level.

For standard MLMC, we have first order weak convergence but the variance of $V_\ell$ for $\ell\geq1$ increases exponentially in $T$, which gives
\begin{equation}
 V_\ell \leq\ \eta_1\, (h_0\, 2^{-\ell})^2\,\e^{\kappa T},
 \label{eq:V_l}
\end{equation}
for some constant $\eta_1>0.$ A good MLMC coupling requires
$\mathrm{C}_0 V_0 > \mathrm{C}_1 V_1,$ and given this condition and $\beta=2,\ \gamma=1,$ the optimal cost is $O(\varepsilon^{-2}\mathrm{C}_0).$ 
The condition $\mathrm{C}_0 V_0 > \mathrm{C}_1 V_1$ requires
\begin{equation}
h_0 = \vartheta_1\, \e^{-\kappa T/2}\ \ \Rightarrow\ \ \mathrm{C}_0= \vartheta_2\, \varepsilon^{-\frac{\kappa}{2\lambda^*}}|\log \varepsilon|,
\label{eq: l_0 std}
\end{equation}
for some $\vartheta_1, \vartheta_2>0.$ The condition $\kappa/\lambda^*<2$ ensures that $h_0$ is greater than the timestep required by the standard Monte Carlo method so additional MLMC levels are required to achieve the desired weak convergence. Therefore, there exists a constant $c_2$ such that
\[\mathrm{C}_{mlmc} \leq\ c_2 \ \varepsilon^{-2-\frac{\kappa}{2\lambda^*}} |\log\varepsilon|.
\]

For the new MLMC with the change of measure, Theorem \ref{thm: level variance} gives
\begin{equation}
 V_\ell \leq\ \eta_2\, (h_0\, 2^{-\ell})^2\,T,
 \label{eq:V_l new}
\end{equation}
for some $\eta_2>0.$ The condition $\mathrm{C}_0 V_0 > \mathrm{C}_1 V_1$ requires
\begin{equation}
h_0 = \vartheta_3\, T^{-1/2},\
\label{eq: l_0 new}
\end{equation}
for some $\vartheta_3>0,$ but the bound in Theorem \ref{thm: level variance} requires the tighter condition
\begin{equation*}
h_0 = \vartheta_4\, T^{-1}\ \ \Rightarrow\ \ \mathrm{C}_0 = \vartheta_5\, |\log \varepsilon|^{2},
\end{equation*}
for some $\vartheta_4, \vartheta_5>0$. Therefore, there exists a constant $c_3$ such that
\[
\mathrm{C}_{com} \leq\ c_3\, \varepsilon^{-2}|\log \varepsilon|^{2}.
\]
\end{proof}

\section{Numerical Results}
\label{Numerical section}
In this section, we present the numerical results for a Lipschitz version of the stochastic Lorenz equation with additive noise:
\begin{equation*}
\begin{aligned}
&f\begin{pmatrix}
x_1\\
x_2\\
x_3
\end{pmatrix}
=\begin{pmatrix}
10(B(x_2)-x_1)\\
(28-x_3) B(x_1)-x_2\\
B(x_1)x_2-\frac{8}{3}x_3
\end{pmatrix}.
\end{aligned}
\end{equation*}
where $B(x)=65x/\max(65,|x|).$
When $|x_1|>65$ and $|x_2|>65,$ we have
\begin{equation*}
\begin{aligned}
&f\begin{pmatrix}
x_1\\
x_2\\
x_3
\end{pmatrix}
=\begin{pmatrix}
650\,\text{sgn}(x_2)-10x_1\\
65\,\text{sgn}(x_1)(28-x_3)-x_2\\
65\,\text{sgn}(x_1)x_2-\frac{8}{3}x_3
\end{pmatrix}.
\end{aligned}
\end{equation*}
Therefore, $f$ satisfy the Lipschitz condition (\ref{Lipschitz condition}) and the dissipativity condition (\ref{eq:dissipativity}). In the region of $|x_1|\leq 65$ and $|x_2|\leq 65,$ which contains the chaotic attractors, this function will retain the chaotic property of the original Lorenz equation. Our interest is to compute $\pi(\varphi),$ where $\varphi(x)= \|x\|$ satisfying the Lipschitz condition. We run 10000 sample paths from $T=0$ to $20$ to get the following results. 

Figure \ref{variancell} is a semi-log plot of the variance on each level as a function of $T$ without the change of measure. The blue to red lines correspond to the variance on each level $\ell=1$ to $8$ with $h_\ell=2^{-\ell}h_0$ and $h_0=2^{-9}:$
$$\mathbb{V}\left[\varphi(\hX^c_T) - \varphi(\hX^f_T)\right] \sim \eta_1\,h_\ell^2\, \e^{\kappa T},$$ 
for some $\eta_1>0,$ which increases exponentially with respect to $T$ and stops increasing when it reaches the decoupling upper bound $\mathbb{V}\left[\varphi(\hX^c_t)\right]+ \mathbb{V}\left[\varphi(\hX^f_t)\right]$ shown in yellow to green lines. In addition, as level increases, the variance decreases at rate $2.$ For $T>10,$ we can see that the standard MLMC on level $\ell=8,$ using $h=2^{-17},$ still can not achieve a good coupling.
\begin{figure}[h]
\includegraphics[width=1\textwidth]{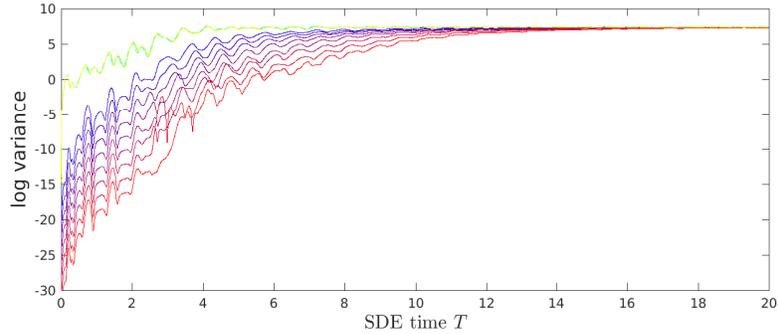}
\caption{Variance for each level without change of measure}\label{variancell}
\end{figure}
In order to see this exponential increase, we plot the log variance on level $8$ using $h=2^{-17}$ with respect to $T$ and the fitted linear function on time interval $[5,10]$, see Figure \ref{varianceLinlog}. The $\kappa$ we fit is $1.36.$ 

\begin{figure}[h]
\subfigure[Linear increase of log variance without change of measure]{
\includegraphics[width=0.45\textwidth]{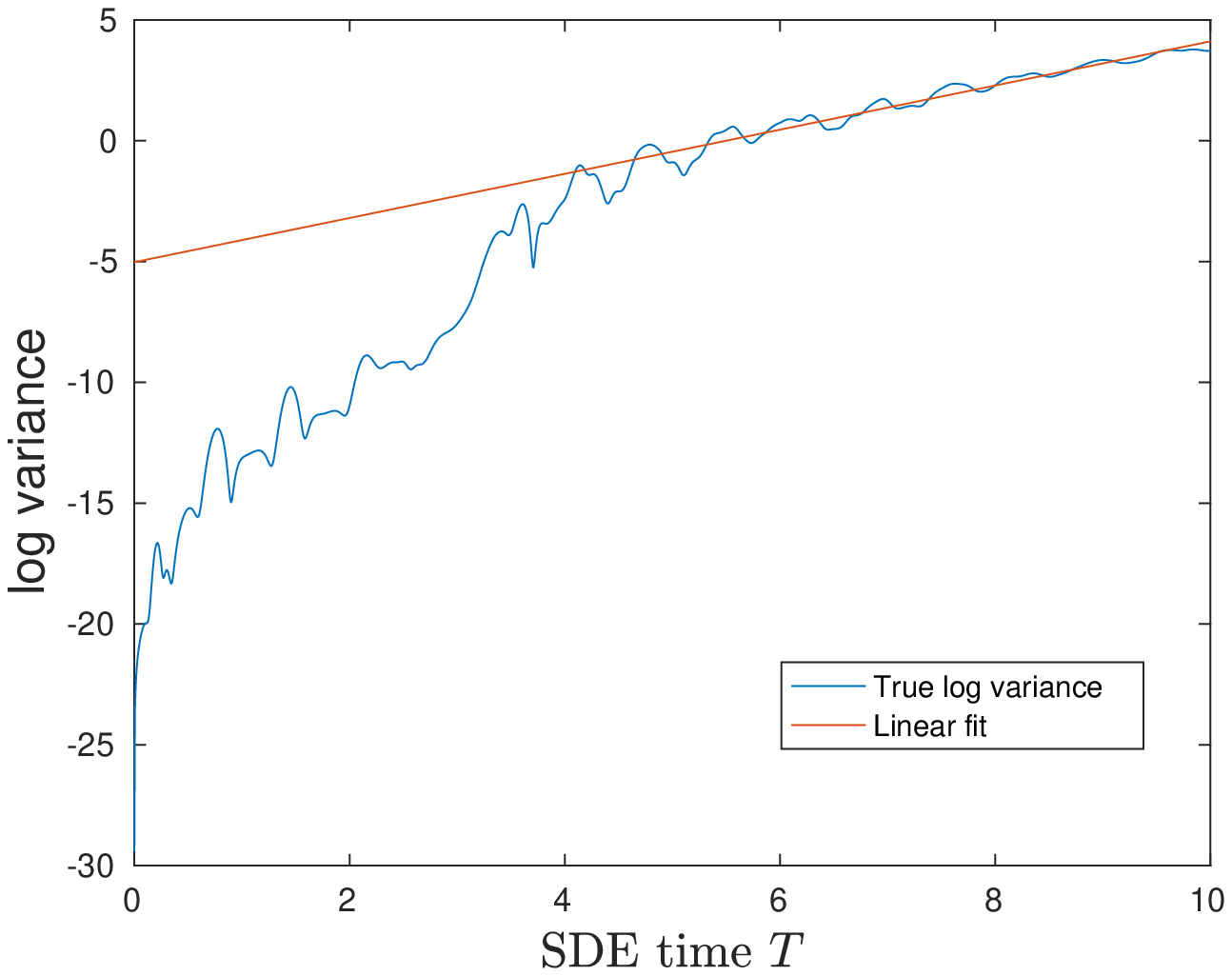}
\label{varianceLinlog}
}
\subfigure[Linear increase of variance with change of measure]{
\includegraphics[width=0.45\textwidth]{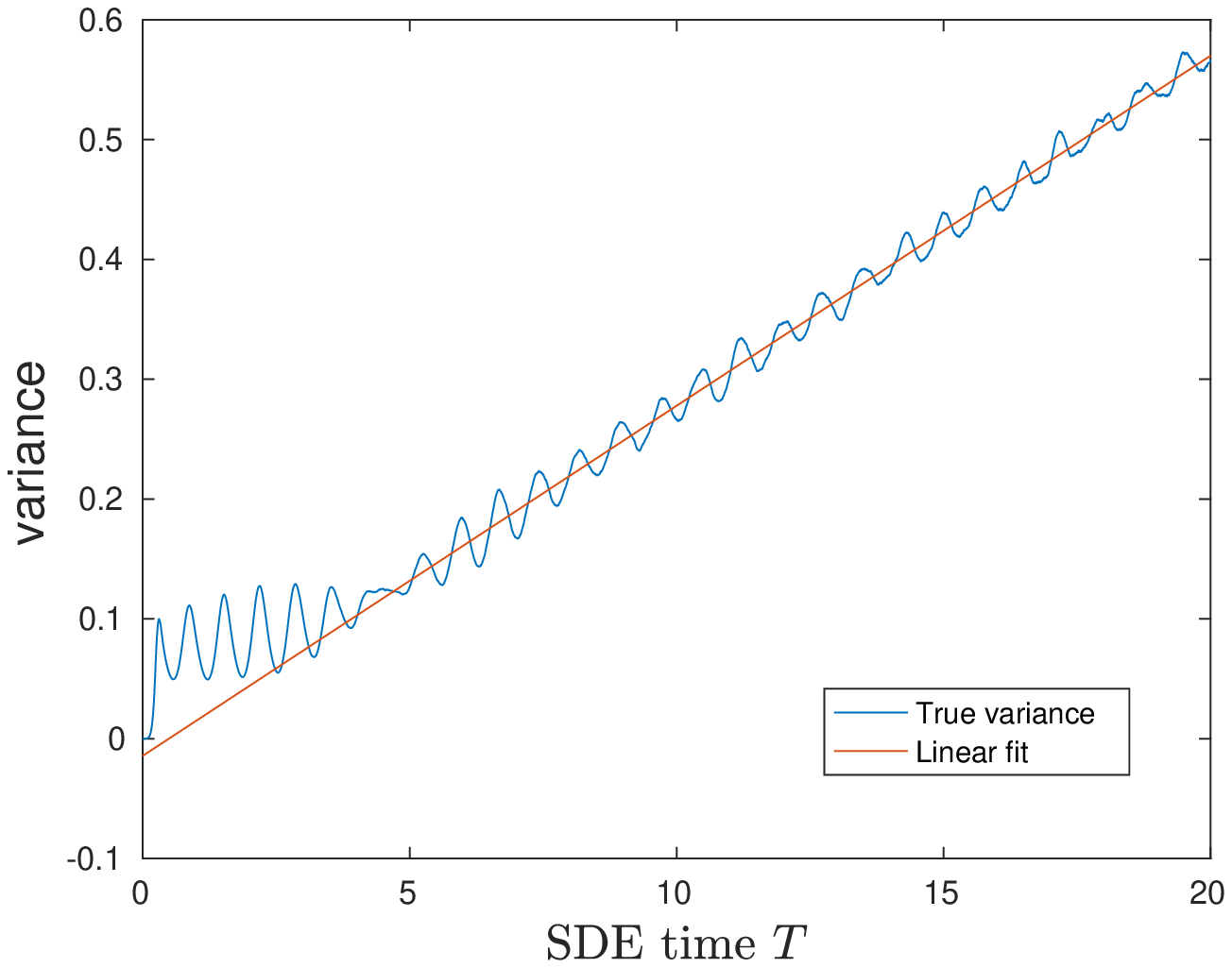}
\label{varianceLinCOM}
}
\caption{Variance on level $8$ with/without change of measure}
\end{figure}

Similarly, for the new MLMC with spring term $S=10,$ Figure \ref{varianceLin} is the semi-log plot of the variance on each level as a function of $T$ with change of measure using same $h_\ell$: 
$$\mathbb{V}\left[\varphi(\hY_T^f)\Rr^f_T-\varphi(\hY_T^c)\Rr^c_T\right] \sim \eta_2\, h_\ell^2\,T ,$$ 
for some $\eta_2>0.$
As the level increases, the variance decreases at a rate $2.$
\begin{figure}[h]
\begin{center}
\includegraphics[width=1\textwidth]{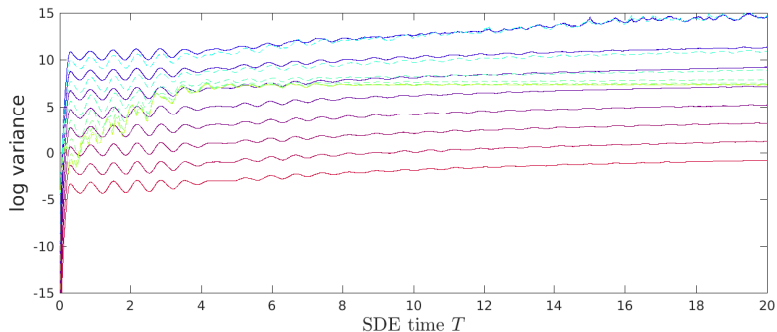}
\end{center}
\caption{Variance for each level with change of measure}
\label{varianceLin}
\end{figure}
In order to see the linear increase in $T$, we plot the variance on level $\ell=8$ with respect to $T$ and the fitted linear function on time interval $[5,20]$, see figure \ref{varianceLinCOM}.

We have investigated and illustrated the dependence of $V_\ell$ on $T$ for both schemes. Next, we investigate the impact of this increase on MLMC schemes, that is the requirement of $h_0$ to achieve a good coupling, that is $V_0>2V_1.$ We plot $\log h_0$  with respect to $T$ in figure \ref{L0T}. The blue line confirms the exponential decrease of $h_0$ with respect to $T$ in \eqref{eq: l_0 std}. The coefficient of the log function fit is $0.49$ which confirms the relationship \eqref{eq: l_0 new}.
\begin{figure}[h]
\center
\includegraphics[width=0.5\textwidth]{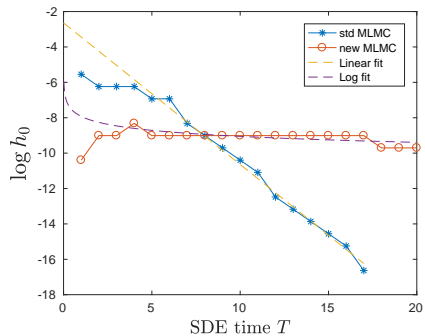}
\caption{The required $h_0$ to achieve a good coupling}
\label{L0T}
\end{figure}

Lastly, we estimate the convergence rate $\lambda^*$ to the invariant measure. Fig \ref{value} plots the function value $\varphi(X_t)$ with respect to time $t$ and its moving upper bound and lower bound. We plot the error bound (the difference between moving upper bound and moving lower bound) in fig \ref{error} and the exponential fit. The fitted $\lambda^*$ is $0.1741.$ Therefore, in this case with $\lambda^*=0.1741$ and $\kappa = 1.3601,$ the standard MLMC fails to achieve any computational savings by Theorem \ref{thm: MLMC}.
\begin{figure}[H]
\subfigure[$ \EE \  \varphi(X_t) $]{
\label{value}
\includegraphics[width=0.45\textwidth]{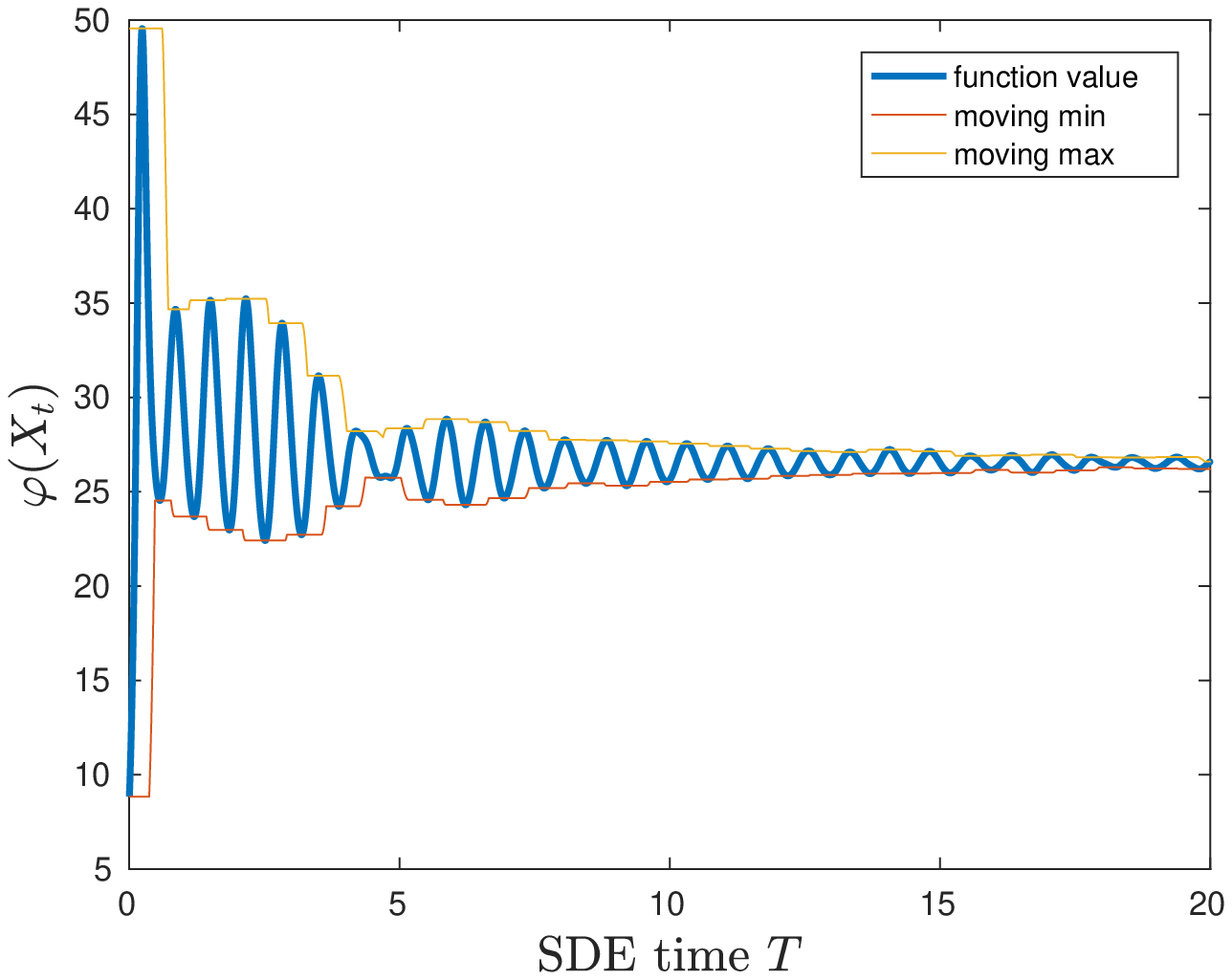}
}
\subfigure[Error bound]{
\label{error}
\includegraphics[width=0.45\textwidth]{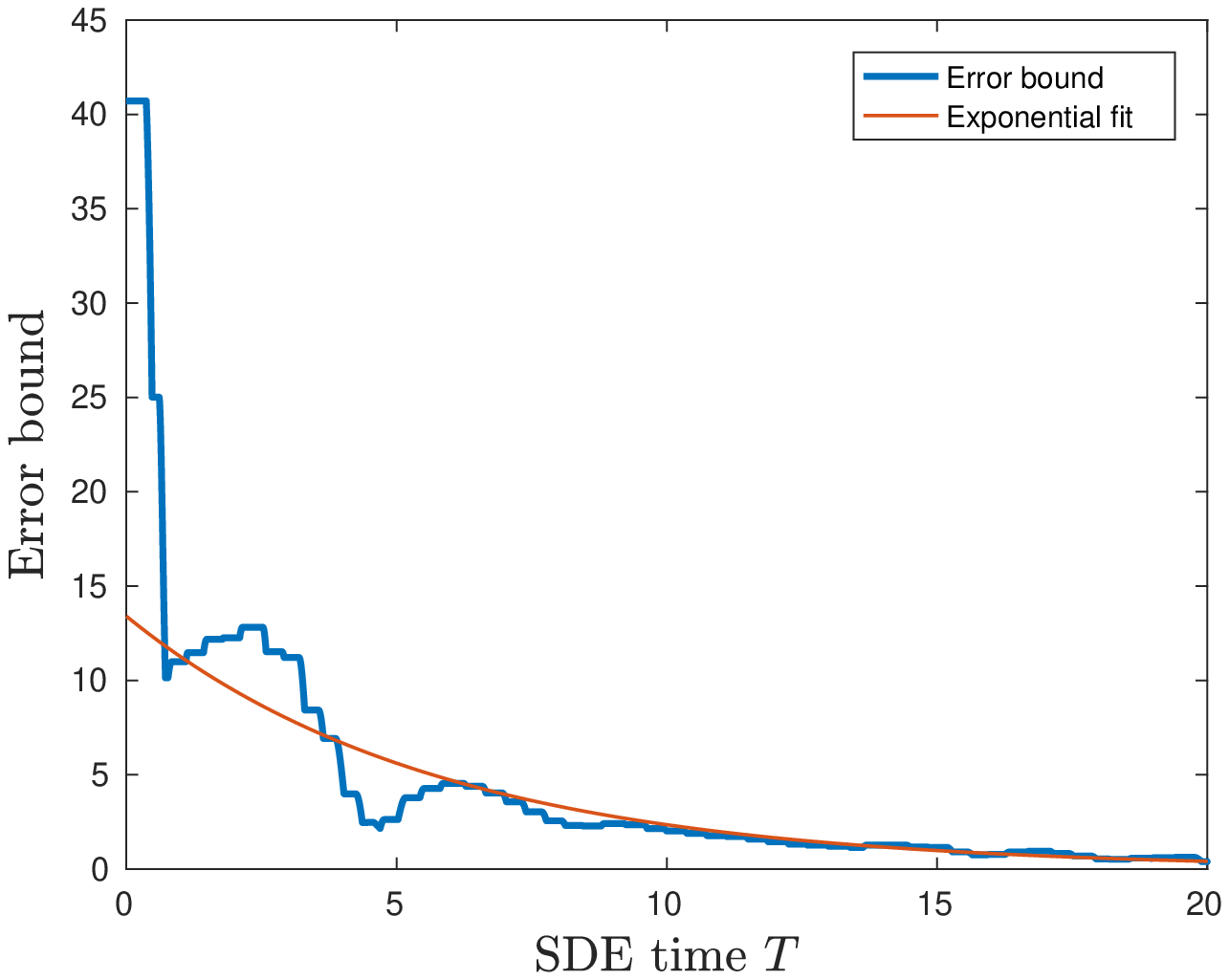}
}
\caption{Estimation of the convergence rate to invariant distribution}
\end{figure}

\section{Extension to non-Lipschitz SDEs}
\label{Non-Lipschitz section}
In this section, we extend this change of measure technique to ergodic SDEs with non-Lipschitz drift using the adaptive timestepping method proposed in \cite{Part2}. Without any proof, we show some numerical experiments results for the SDE with a double-well potential energy and the stochastic Lorenz equation.
\subsection{Double-well potential energy} We consider
\begin{equation}
\D X_t = (2X_t - \frac{1}{2}X_t^3)\,\D t + \D W_t.
\label{DWsde}
\end{equation}
The probability density function of its invariant distribution is
\[\frac{\exp(2x^2-\frac{1}{4}x^4)}{\int_{-\infty}^{\infty} \exp(2x^2-\frac{1}{4}x^4) \D x},\]
and it has two different wells at $x=\pm 2.$
This SDE satisfies the dissipativity condition (\ref{eq:dissipativity}) and one-sided Lipschitz condition (\ref{eq:one_sided_Lipschitz3}) with $\lambda=2$ but the drift is non-globally Lipschitz. For the standard MLMC scheme, the issue is that the fine and coarse paths may diverge to different wells, which can result in a large variance and high kurtosis. Using the change of measure technique can reduce the divergence and then improve the efficiency.

We simulate the SDE with initial value $x_0=0$ to time $T=5,$ and use the adaptive function:
\[
h^\delta (x) = \frac{\max(1,|x|)}{8\max(1,|2x-\frac{1}{2}x^3|)}\ \delta,
\]
with $\delta=2^{-\ell}$ for each level $\ell.$ We compare three different schemes:
\begin{itemize}
\item standard MLMC with adaptive timestep.
\item MLMC with adaptive timestep and change of measure with constant spring coefficient $S=1.$
\item MLMC with adaptive timestep and change of measure with adaptive spring coefficient $$S=\max(0,2-1.5x^2).$$
\end{itemize}
The second scheme uses $S=1$ following the suggestion of Theorem \ref{thm: level variance}. The third scheme improve on the second by choosing adaptive $S$ and avoiding unnecessary spring term, reducing the variance without losing the control on divergence. By doing first order Taylor expansion on \eqref{adaptive S}, we choose $S=\max(0,f'(x))$ to deal with the divergence locally.

We run 10000 samples for each level $\ell$ for the three schemes. The numerical results are shown in Figure \ref{DW_sum}.
\begin{figure}[h]
\center
\includegraphics[width=0.8\textwidth]{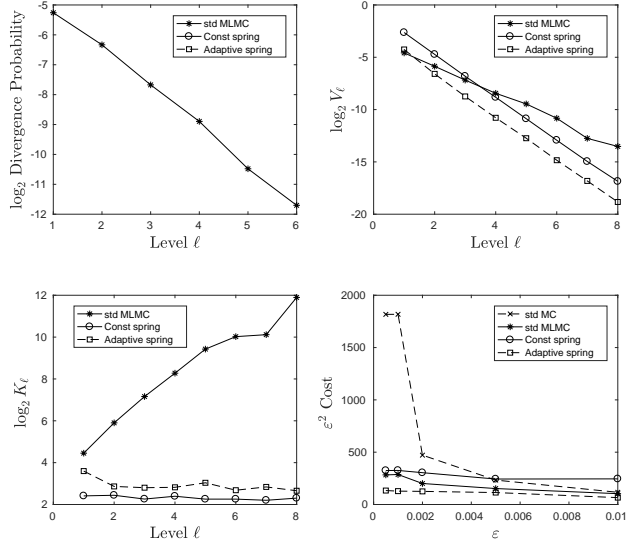}
\caption{MLMC convergence test for double-well potential energy}
\label{DW_sum}
\end{figure}

The top left figure plots the divergence probability with respect to the level $\ell,$ where the divergence probability is defined as 
\[\EE\left[ \mathbbm{1}_{\|\hX^f_T-\hX^c_T\|>1}\right]\ = \ \mathbb{P}\left[\|\hX^f_T-\hX^c_T\|>1\right].\]  The probability decreases as $\ell$ increases since the timestep $h_\ell$ is smaller and the difference between fine and coarse path decreases. The decrease rate we fit is
\[\mathbb{P}\left[\|\hX^f_T-\hX^c_T\|>1\right] \sim O(h_\ell^{1.28}).\]
The two schemes with change of measure have zero divergence on all levels. 

The top right figure plots the variance of corrections $V_\ell$ with respect to level $\ell.$ The $V_\ell$ of the two schemes with change of measure decrease at the similar rate $2$ while the standard MLMC has a slower rate of approximately $1.28$ since the divergence of the fine and coarse paths dominated the variance.
The scheme with the adaptive spring coefficient has lower $V_\ell$ than the scheme with constant spring coefficient since the unnecessary spring will increase the variance of the Radon-Nikodym derivative. 

The bottom left figure shows the log kurtosis with respect to level $\ell.$ The kurtosis of standard MLMC will increase exponentially while the kurtosis of the schemes with change of measure will stay constant. Similar intuitive explanation applies here. The divergence samples again dominate the 4th moment and then the kurtosis on each level
\[
K_\ell \sim \frac{\EE\left[ \|\hX^f_T-\hX^c_T\|^4\right] }{\EE\left[ \|\hX^f_T-\hX^c_T\|^2\right]^2} \sim h_\ell^{-1.28}.
\] 
The rate of increase in the figure is 1.06 which is quite close to the rate of decrease of the divergence probability.

The bottom right figure plots the costs of the three schemes together with the standard Monte Carlo method with respect to $\varepsilon.$ The costs of all the MLMC schemes are $O(\varepsilon^{-2})$ while the standard MC is $O(\varepsilon^{-3})$ and the scheme with adaptive spring has the lowest cost.

Overall, the new MLMC schemes with change of measure perform better especially the one with adaptive spring. They can not only keep the kurtosis constant but also reduce the variance and hence the total computational cost. 

\subsection{Stochastic Lorenz equation}
This is a three-dimensional system modelling convection rolls in the atmosphere

\begin{equation}
\begin{aligned}
&f\begin{pmatrix}
x_1\\
x_2\\
x_3
\end{pmatrix}
=\begin{pmatrix}
10(x_2-x_1)\\
x_1(28-x_3) -x_2\\
x_1x_2-\frac{8}{3}x_3
\end{pmatrix}.
\end{aligned}
\label{Lorenz sde}
\end{equation}

This SDE does not satisfy the dissipativity condition (\ref{eq:dissipativity}) and one-sided Lipschitz condition (\ref{eq:one_sided_Lipschitz3}), and is more chaotic compared with the truncated Lipschitz version in previous section.

We simulate the SDE with initial value $x_0=[0,0,0]$ to time $T=10,$ and use the adaptive function:
\[
h^\delta(x) = \frac{\max(100,\|x\|^2)}{2^{11}\max(100,\|f(x)\|^2)}\delta,
\]
with $\delta = 2^{-\ell}$ for each level $\ell.$  We compare two different schemes:
\begin{itemize}
\item standard MLMC with adaptive timestep.
\item MLMC with adaptive timestep and change of measure with constant spring coefficient $S=10.$
\end{itemize}
A possible third scheme is the scheme with adaptive spring which requires us to calculate the largest positive eigenvalue of the Jacobian matrix $\frac{\partial f}{\partial x}.$

We run 10000 samples for each level $\ell$ for two schemes. The numerical results are shown in the figure \ref{Lorenz full}.

\begin{figure}[h]
\center
\includegraphics[width=0.8\textwidth]{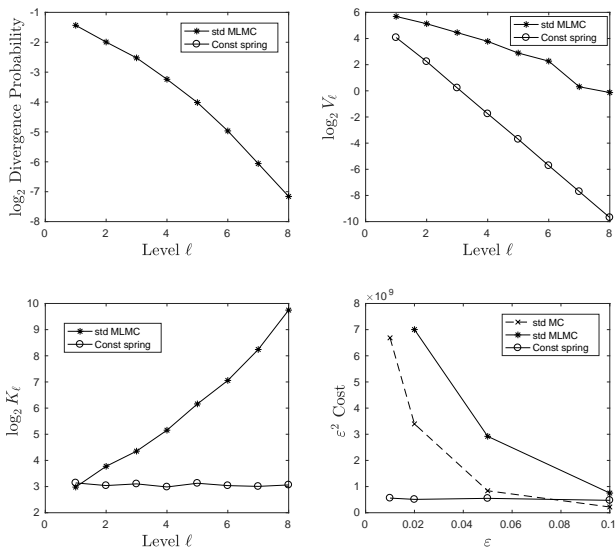}\label{Lorenz full}
\caption{MLMC convergence test for Lorenz equation}
\end{figure}

Similarly, the top left figures shows that the change of measure technique can greatly reduce the ratio of divergence $\EE\left[ \mathbbm{1}_{\|\hX^f_T-\hX^c_T\|>10}\right]$ and actually no divergence occurs in this numerical experiment for all the levels. The rate of decrease for standard MLMC is $0.82.$ 
 
The top right figure illustrates the variance reduction of the change of measure technique, and the rate of decrease of the variance for level corrections $V_\ell$ is approximately 2 for change of measure and $0.83$ for standard MLMC which is similar to the rate of decrease of divergence probability.
 
The bottom left plot shows that the kurtosis of standard MLMC increases exponentially as level $\ell$ increases while the kurtosis of change of measure remains constant. The increase rate is $0.96$ which is close to the decrease rate of divergence rate.

The last bottom right plot implies that the total computational cost is $O(\varepsilon^{-2})$ for the MLMC with change of measure and $O(\varepsilon^{-3})$ for the standard Monte Carlo method. The $O(\varepsilon^{-3})$ computational cost for standard MLMC is worse than the theoretical results due to the high kurtosis and large variance $V_\ell$ and it is already quite hard to get the result for $\varepsilon=0.01$ in a reasonable computational time.

\section{Proofs}
For simplicity of the proof, we introduce the notation $a(h)\lesssim b(h)$ which means there exists a constant $\tilde{h}_0>0$ such that $a(h)\leq b(h),$ $\forall\ 0<h<\tilde{h}_0,$ where $\tilde{h}_0$ is allowed to depend on constants such as $S,$ $K,$ $\tilde{\alpha},$ $\tilde{\beta},$ $f(0)$ but not on stochastic samples $\omega$ or Brownian paths.

Note that for all $\delta>0,$ we have
$1/(1-Sh) \lesssim 1/(1-2Sh) \lesssim 1+2Sh+\delta h \lesssim 2.$

\subsection{Theorem \ref{Theorem: stability of path}}
\label{proof 1}
\begin{proof}
The proof is given for $p\!\geq\! 4$; the result for 
$1\!\leq\! p \!< \! 4$ follows from H{\"o}lder's inequality. 
We start our proof by analyzing the numerical paths step by step.
When $t=t_0=0,$ the two numerical paths are both at initial point $x_0,$ i.e. $\hY_{t_0}^f=\hY_{t_0}^c=x_0.$

For the odd time point $t_{2n+1}$ for $n\geq 0$, 
\begin{eqnarray*}
\label{eq: odd update 1}
 \hY_{t_{2n+1}}^c & = &\hY_{t_{2n}}^c+ S(\hY_{t_{2n}}^f-\hY_{t_{2n}}^c)h + f(\hY_{t_{2n}}^c) h + \Delta W_{2n}, \\
\label{eq: odd update 2}
\hY_{t_{2n+1}}^f & = &\hY_{t_{2n}}^f+ S(\hY_{t_{2n}}^c-\hY_{t_{2n}}^f)h + f(\hY_{t_{2n}}^f) h + \Delta W_{2n}. 
\end{eqnarray*} 
Squaring both sides gives
\begin{eqnarray*}
\|\hY_{t_{2n+1}}^c\|^2 &=&\left\| Sh\,\hY_{t_{2n}}^f+(1-Sh)\left(\hY_{t_{2n}}^c + \frac{f(\hY_{t_{2n}}^c) h + \Delta W_{2n}}{1-Sh} \right)\right\|^2.
\end{eqnarray*}
Due to the convexity of $x^2$ that, for any $0\leq \xi\leq 1,$
\begin{eqnarray*}
\|\xi A+(1-\xi)B\|^2 \leq \xi \|A\|^2 +(1-\xi)\|B\|^2,
\end{eqnarray*}
provided $h<1/S,$ we can choose $\xi = Sh$ to get
\begin{eqnarray*}
\|\hY_{t_{2n+1}}^c\|^2 &\lesssim & Sh\|\hY_{t_{2n}}^f\|^2+(1-Sh)\|\hY_{t_{2n}}^c\|^2+ 4\,\|\Delta W_{2n}\|^2+2\,\langle \hY_{t_{2n}}^c, f(\hY_{t_{2n}}^c) \rangle h\\
&&+\, 4\,\|f(\hY_{t_{2n}}^c)\|^2 h^2  + 2\,\langle \hY_{t_{2n}}^c, \Delta W_{2n}\rangle. 
\end{eqnarray*}
Due to the Lipschitz condition \eqref{Lipschitz condition}, 
\[
\|f(\hY_{t_{2n}}^c)\|^2  h^2\lesssim \gamma\, h ( \| \hY_{t_{2n}}^c \|^2 + 1 )
\]
for any $\gamma>0$.  Combining this with dissipativity condition (\ref{eq:dissipativity}), we obtain, for some fixed $\alpha \in(0,\tilde{\alpha})$ and $\beta \in(\tilde{\beta},\infty),$
\begin{eqnarray}
\label{ineq: coarse path}
\|\hY_{t_{2n+1}}^c\|^2 &\lesssim & Sh\|\hY_{t_{2n}}^f\|^2+(1-Sh-2\alpha h)\|\hY_{t_{2n}}^c\|^2+ 4\,\|\Delta W_{2n}\|^2   + 2\beta h\nonumber\\
&& + 2\,\langle \hY_{t_{2n}}^c, \Delta W_{2n}\rangle.
\end{eqnarray}
Similarly, we have
\begin{eqnarray}
\label{ineq: fine path}
 \|\hY_{t_{2n+1}}^f\|^2 &\lesssim & Sh\|\hY_{t_{2n}}^c\|^2+(1-Sh-2\alpha h)\|\hY_{t_{2n}}^f\|^2+ 4\,\|\Delta W_{2n}\|^2 + 2\beta h \nonumber\\
 &&+\, 2\,\langle \hY_{t_{2n}}^f, \Delta W_{2n}\rangle .
\end{eqnarray}

For the even point $t_{2n+2}$ for $n\geq 0,$
\begin{eqnarray*}
\hY_{t_{2n+2}}^c &=&\hY_{t_{2n}}^c+ S(\hY_{t_{2n}}^f-\hY_{t_{2n}}^c)2h + f(\hY_{t_{2n}}^c) 2h + \Delta W_{2n} +\Delta W_{2n+1},\\
 \hY_{t_{2n+2}}^f &=&\hY_{t_{2n+1}}^f+ S(\hY_{t_{2n+1}}^c-\hY_{t_{2n+1}}^f)h + f(\hY_{t_{2n+1}}^f) h +\Delta W_{2n+1}.
\end{eqnarray*}
Using the same approach and choosing $\xi=2Sh$ provided $2Sh<1$, we get
\begin{eqnarray*}
\|\hY_{t_{2n+2}}^c\|^2  &\lesssim &  2Sh\|\hY_{t_{2n}}^f\|^2+(1-2Sh-4\alpha h)\|\hY_{t_{2n}}^c\|^2+4\|\Delta W_{2n}\!+\!\Delta W_{2n+1}\|^2 \\
&& + 4\beta h +\, 2\,\langle \hY_{t_{2n}}^c, \Delta W_{2n}\!+\!\Delta W_{2n+1}\rangle,
\end{eqnarray*}
and
\begin{eqnarray*}
\|\hY_{t_{2n+2}}^f\|^2 &\lesssim & Sh\|\hY_{t_{2n+1}}^c\|^2+(1-Sh-2\alpha h)\|\hY_{t_{2n+1}}^f\|^2+ 4\,\|\Delta W_{2n+1}\|^2 + 2\beta h \\
&& + 2\,\langle \hY_{t_{2n+1}}^f, \Delta W_{2n+1}\rangle .
\end{eqnarray*}

Therefore, for any fixed $\gamma\in(0,\alpha),$ we have
\begin{eqnarray*}
\|\hY_{t_{2n+2}}^c\|^2\!+\!\|\hY_{t_{2n+2}}^f\|^2 \!\!&\!\!\lesssim\!\! &\!\! (1\!-\!4\gamma h)(\|\hY_{t_{2n}}^c\|^2\! +\!\|\hY_{t_{2n}}^f\|^2)
\!+\! 12(\|\Delta W_{2n}\|^2\!+\!\|\Delta W_{2n+1}\|^2)\\
&&+8\beta h+\, 2\,\e^{-4\gamma h} \langle \phi_{t_{2n}}, \Delta W_{2n}\rangle + 2\,\e^{-2\gamma h} \langle \phi_{t_{2n+1}}, \Delta W_{2n+1}\rangle,
\end{eqnarray*}
where for $n\in[0,N/2-1],$
\begin{eqnarray*}
\e^{-4\gamma h}\phi_{t_{2n}} = (1+Sh)\hY_{t_{2n}}^c +(1-Sh-2\alpha h)\hY_{t_{2n}}^f, \ \ \
\e^{-2\gamma h}\phi_{t_{2n+1}} = \hY_{t_{2n}}^c+\hY_{t_{2n+1}}^f.
\end{eqnarray*}
Since $1-4\gamma h \leq \e^{-4\gamma h}$ and $\e^{4\gamma h}\lesssim 2,$ we multiply by $\e^{2\gamma t_{2n+2}}$ on both sides to obtain
\begin{eqnarray*}
\e^{2\gamma t_{2n+2}}(\|\hY_{t_{2n+2}}^f\|^2\!+\!\|\hY_{t_{2n+2}}^c\|^2)  &\lesssim & \e^{2\gamma t_{2n}} (\|\hY_{t_{2n}}^f\|^2\!+\!\|\hY_{t_{2n}}^c\|^2)+\,16\beta \e^{2\gamma t_{2n}} h\\
&&\!\!\!\!\!\!\!\!\!\!\!\!\!\!\!\!\!\!\!\!\!\!\!\!\!\!\!\!\!\!\!\!\!\!\!\!\!\! +24\,  \e^{2\gamma t_{2n}}(\|\Delta W_{2n}\|^2\! +\! \|\Delta W_{2n+1}\|^2)+ 2\, \e^{2\gamma t_{2n}}\langle \phi_{t_{2n}}, \Delta W_{2n}\rangle\\
&& + 2\, \e^{2\gamma t_{2n+1}}\langle \phi_{t_{2n+1}}, \Delta W_{2n+1}\rangle.
\end{eqnarray*}
Summing over multiple timesteps gives
 \begin{eqnarray} 
 \label{eq: even step}
\e^{2 \gamma  t_{2n}}(\|\hY_{t_{2n}}^f\|^2+\|\hY_{t_{2n}}^c\|^2)  &\lesssim &   (\|\hY_{t_0}^f\|^2+\|\hY_{t_0}^c\|^2) 
+24 \sum_{k=0}^{2n-1} \e^{2\gamma t_{k}}\|\Delta W_{k}\|^2 
\nonumber\\
&& \!\!\!\!\!\!\!\! +16\beta  \sum_{k=0}^{n-1}\e^{2\gamma t_{2k}} h 
+\, 2\, \sum_{k=0}^{2n-1} \e^{2\gamma t_k}\langle \phi_{t_k}, \Delta W_{k}\rangle .
\end{eqnarray}
For odd time points, combining \eqref{ineq: coarse path} and \eqref{ineq: fine path}, by Cauchy-Schwarz inequality and Young's inequality, there exist constants $\alpha_1>1,\ \beta_1>\max(1,\alpha_1\beta)$ such that
\begin{eqnarray*}
\|\hY_{t_{2n+1}}^f\|^2+\|\hY_{t_{2n+1}}^c\|^2 &\lesssim & (1-2\alpha h)(\|\hY_{t_{2n}}^f\|^2+\|\hY_{t_{2n}}^c\|^2)+ 8\,\|\Delta W_{2n}\|^2 + 4\beta h\nonumber\\
&&+\, 2\langle \hY_{t_{2n}}^c, \Delta W_{2n}\rangle  + 2\langle \hY_{t_{2n}}^f, \Delta W_{2n}\rangle \\
&&\!\!\!\!\!\!\!\!\!\!\!\!\!\!\!\!\!\!\!\! \lesssim  (1-2\gamma h) \left(\alpha_1(\|\hY_{t_{2n}}^f\|^2+\|\hY_{t_{2n}}^c\|^2 +12\,\|\Delta W_{2n}\|^2 )+ 4\, \beta_1 h\right).
\end{eqnarray*}
Multiplying by $\e^{2\gamma t_{2n+1}}$ on both sides and using the \eqref{eq: even step} gives
\begin{eqnarray}
 \label{eq: odd step}
\e^{2\gamma t_{2n+1}}(\|\hY_{t_{2n+1}}^f\|^2 + \|\hY_{t_{2n+1}}^c\|^2) &\lesssim & \alpha_1  (\|\hY_{t_0}^f\|^2 + \|\hY_{t_0}^c\|^2) +  24\,\alpha_1  \sum_{k=0}^{2n} \e^{2\gamma t_{k}}\|\Delta W_{k}\|^2 \nonumber\\
 &&\!\!\!\!\!\!\!\!\!\!\!\!\!\!+  16\,\beta_1 \sum_{k=0}^{n}\e^{2\gamma t_{2k}} h 
+ 2\,\alpha_1 \sum_{k=0}^{2n-1} \e^{2\gamma t_k}\langle \phi_{t_k}, \Delta W_{k}\rangle .
\end{eqnarray}
Then, combing \eqref{eq: even step} and \eqref{eq: odd step}, raising both sides to power $p/2$, taking the supremum over $n\in[0,N]$ and taking expectation on the both sides, by using Jensen's inequality, we have
\begin{eqnarray*}
\EE\left[\sup_{0\leq n\leq N} \e^{\gamma p t_{n}}\left(\|\hY^f_{t_{n}}\|^2+\|\hY^c_{t_{n}}\|^2\right)^{p/2}\right]  & \lesssim &   4^{p/2-1}(24\,\alpha_1\beta_1)^{p/2} (I_1 + I_2 + I_3 +I_4 ),
\end{eqnarray*}
where
\begin{eqnarray*}
I_1 &=& \EE\left[(\|\hY_{t_0}^f\|^2+\|\hY_{t_0}^c\|^2)^{p/2} \right] = 2^{p/2} \|x_0\|^{p/2},\ \ \ \ \ 
I_2 =\ \left|\  \sum_{k=0}^{N-1}\e^{2\gamma t_{k}} h\ \right|^{p/2},\\
I_3&= &\EE\left[\left| \sum_{k=0}^{N-1} \e^{2\gamma t_{k}} \|\Delta W_{k}\|^2 \right|^{p/2} \right],\ 
I_4 = \EE\left[\ \sup_{1\leq n \leq N/2} \left|  \sum_{k=0}^{2n-1} \e^{2\gamma t_k}\langle \phi_{t_k}, \Delta W_{k}\rangle\right|^{p/2} \right].
\end{eqnarray*}
We will bound these four parts separately. $I_1$ is a constant.
For $I_2,$ we have
\begin{eqnarray*}
I_2& \leq & \left|\  \int_0^T \e^{2\gamma t}\ \D t\right|^{p/2} \leq 
 \e^{\gamma p T} / (2\gamma)^{p/2}.
\end{eqnarray*}
Next, if $q_k$, for  $k=1,\ldots,n$ is an arbitrary discrete probability distribution, and $b_k$, for $ k=1,\ldots,n$ is a set of scalar values, then for any $p>1$
Jensen's inequality gives
\[ \left| \sum_{k=1}^n q_k b_k \right|^p \leq \sum_{k=1}^n
q_k |b_k|^p.\]
If $a_k, k=1,\ldots,n$ is a set of positive scalar values,
then setting $q_k = a_k / \sum_{k'=1}^n a_{k'} $ gives
\[
\left| \sum_{k=1}^n a_k b_k \right|^p \leq \left|\sum_{k=1}^n a_k \right|^{p-1} \sum_{k=1}^n
a_k |b_k|^p.
\]
For $I_3,$ using this inequality, we obtain
\begin{eqnarray*}
I_3 &=&\EE\left[\ \left|  \sum_{k=0}^{N-1} \e^{2\gamma t_{k}}h\frac{\|\Delta W_{k}\|^2}{h}\right|^{p/2} \right]
\leq  \left|\sum_{k=0}^{N-1} \e^{2\gamma t_{k}}h\right|^{p/2-1}\EE\left[ \sum_{k=0}^{N-1} \e^{2\gamma t_{k}}h \frac{\|\Delta W_{k}\|^p}{h^{p/2}}\right] \\
&\leq& c_p\,\e^{\gamma p T} / (2\gamma)^{p/2}.
\end{eqnarray*}
where $c_p$ is defined by 
$c_p=\EE\left[\|\Delta W_{k}\|^p/h^{p/2}\right]\leq d^{p/2}\,p!! \leq d^{p/2}p^{p/2}.$\\
For $I_4,$ we rewrite the summation as an It{\^o} integral and then by the Burkholder-Davis-Gundy inequality in \cite{BY82}, there exists a positive constant $C_{\scaleto{\mathrm{BDG}}{4pt}}$ independent of $p$ such that
\begin{eqnarray*}
I_4 &\leq & \EE\left[\sup_{0\leq t \leq T} \left| \int_0^t \e^{2\gamma \floor{s/h}h}\langle \phi_{\floor{s/h}h}, \D  W_{s}\rangle \right|^{p/2}\right]\\
&\leq & (C_{\scaleto{\mathrm{BDG}}{4pt}}\,p)^{p/4}\ \EE\left[\ \left| \sum_{k=0}^{N-1} \e^{4\gamma t_k}\|\phi_{t_k}\|^2h\right|^{p/4}\ \right],
\end{eqnarray*}
where, by Young's inequality,
\begin{eqnarray*}
\|\phi_{t_{2k}}\|^2 
\lesssim  8(\|\hY_{t_{2k}}^c\|^2+\|\hY_{t_{2k}}^f\|^2),\ \ \ \ \ 
\|\phi_{t_{2k+1}}\|^2
\lesssim  4\,(\|\hY_{t_{2k}}^c\|^2  + \|\hY_{t_{2k+1}}^f\|^2).
\end{eqnarray*}
Then by Jensen's inequality and Young's inequality, for arbitrary $\zeta>0,$ we have
\begin{eqnarray*}
I_4 \!\!\!\!&\leq & \!\!\!\!(12 C_{\scaleto{\mathrm{BDG}}{4pt}}p)^{p/4}\left|\sum_{k=0}^{N-1}\e^{2\gamma t_{k}} h \right|^{p/4-1}\!\!\!\!\EE \left[\sum_{k=0}^{N-1}\e^{2\gamma t_{k}} h \!\sup_{0\leq n\leq N} \e^{\gamma p t_{n}/2} (\|\hY_{t_{n}}^c\|^2 \! +\! \|\hY_{t_{n}}^f\|^2)^{p/4} \right]\\
&\leq &\!\!\!\! \frac{1}{4\zeta}\, \EE\left[\sup_{0\leq n\leq N} \e^{\gamma p t_{n}}\left(\|\hY^f_{t_{n}}\|^2\! +\! \|\hY^c_{t_{n}}\|^2\right)^{p/2}\right] + \zeta\,\left(\frac{6\,C_{\scaleto{\mathrm{BDG}}{4pt}}}{\gamma}\right)^{p/2} p^{p/2}\e^{\gamma p T}.
\end{eqnarray*}

Finally, combining all the estimates above and choosing $\zeta=4^{p/2-1}(24\alpha_1\beta_1)^{p/2}  ,$ there exist a constants $C_{(2)}>0$ such that 
\begin{eqnarray*}
\EE\left[\sup_{0\leq n\leq N} \e^{\gamma p t_{n}}\left(\|\hY^f_{t_{n}}\|^2 \!\!+\!\!\|\hY^c_{t_{n}}\|^2\right)^{p/2}\right] &\lesssim &  C_{(2)}^p\, p^{p/2} \e^{\gamma p T},
\end{eqnarray*}
which implies that there exists constant $C_{(1)}>0$ such that, for any $ 0<h\leq C_{(1)},$
\[
\sup_{0\leq n \leq N} \EE\left[ \| \widehat{Y}^f_{t_n} \|^p \right] 
\leq C_{(2)}^p\, p^{p/2}, ~~
\sup_{0\leq n \leq N} \EE\left[ \| \widehat{Y}^c_{t_n} \|^p \right] 
\leq C_{(2)}^p\, p^{p/2}.
\]  
\end{proof}

\subsection{Theorem \ref{thm: strong error}}
\label{proof 2}
\begin{proof}
The proof is given for $p\!\geq\! 4$; the result for 
$1\!\leq\! p \!< \! 4$ follows from H{\"o}lder's inequality.

The different updates on odd and even time points give
\begin{eqnarray*}
\hY_{t_{2n+1}}^f\!-\!\hY_{t_{2n+1}}^c \!\!\!\!&=&\!\!\!\! (1-2Sh)(\hY_{t_{2n}}^f- \hY_{t_{2n}}^c)+ (f(\hY_{t_{2n}}^f) - f(\hY_{t_{2n}}^c))h,\\
\hY_{t_{2n+2}}^f\!-\! \hY_{t_{2n+2}}^c \!\!\!\!&=&\!\!\!\! (1\!-\!Sh)(\hY_{t_{2n+1}}^f \!-\!\hY_{t_{2n+1}}^c)-Sh(\hY_{t_{2n}}^f\!-\!\hY_{t_{2n}}^c)+(f(\hY_{t_{2n+1}}^f)\!-\! f(\hY_{t_{2n}}^c))h,
\end{eqnarray*}
and then
\begin{eqnarray*}
\hY_{t_{2n+2}}^f\! - \! \hY_{t_{2n+2}}^c
 \!\!\!\!&=&\!\!\!\!  (1-4Sh+2S^2h^2)(\hY_{t_{2n}}^f- \hY_{t_{2n}}^c) + (1-Sh)(f(\hY_{t_{2n}}^f) - f(\hY_{t_{2n}}^c))h \\
&& +\ (f(\hY_{t_{2n+1}}^f)- f(\hY_{t_{2n}}^f))h.
\end{eqnarray*}
Taking the square of both sides gives
\begin{eqnarray*}
\|\hY_{t_{2n+2}}^f- \hY_{t_{2n+2}}^c\|^2  &=&  (1-4Sh+2S^2h^2)^2 \|\hY_{t_{2n}}^f- \hY_{t_{2n}}^c\|^2 \\
&&\!\!\!\!\!\!\!\!\!\!\!\!\!\!\!\! +(1-Sh)^2\|f(\hY_{t_{2n}}^f) - f(\hY_{t_{2n}}^c)\|^2h^2 +\|f(\hY_{t_{2n+1}}^f)- f(\hY_{t_{2n}}^f)\|^2 h^2\\
&&\!\!\!\!\!\!\!\!\!\!\!\!\!\!\!\!\!\! + 2(1-4Sh+2S^2h^2)(1-Sh)\langle \hY_{t_{2n}}^f- \hY_{t_{2n}}^c, f(\hY_{t_{2n}}^f) - f(\hY_{t_{2n}}^c) \rangle h\\
&& +2(1-4Sh+2S^2h^2) \langle \hY_{t_{2n}}^f- \hY_{t_{2n}}^c,f(\hY_{t_{2n+1}}^f)- f(\hY_{t_{2n}}^f) \rangle h
\\
&& +2(1-Sh) \langle f(\hY_{t_{2n}}^f) - f(\hY_{t_{2n}}^c), f(\hY_{t_{2n+1}}^f)- f(\hY_{t_{2n}}^f) \rangle h^2.
\end{eqnarray*}
Then provided $S>\lambda/2$, the Assumption \ref{assp:enhanced_one_sided_Lipschitz}, Lipschitz condition (\ref{Lipschitz condition}), Cauchy-Schwarz inequality and Young's inequality imply,  for any fixed $\gamma\in(0,2S-\lambda),$ 
\begin{eqnarray}
\label{error:2}
\|\hY_{t_{2n+2}}^f- \hY_{t_{2n+2}}^c\|^2 
&\lesssim& (1-4\gamma h) \|\hY_{t_{2n}}^f- \hY_{t_{2n}}^c\|^2 +2K^2\|\hY_{t_{2n+1}}^f- \hY_{t_{2n}}^f\|^2 h^2\nonumber\\
&& \!\!\!\!\!\!\!\!\!\!\!\!\!\!\!\!\!\!\!\!\!\!\!\!\!\!\!\!\!+\,2(1-4Sh+2S^2h^2) \langle \hY_{t_{2n}}^f- \hY_{t_{2n}}^c,f(\hY_{t_{2n+1}}^f)- f(\hY_{t_{2n}}^f) \rangle h.
\end{eqnarray}
Following this estimate, we use two different approach to get different upper bounds.

\textbf{First}, we continue to use Young's inequality and  Lipschitz condition (\ref{Lipschitz condition}) to get
\begin{eqnarray*}
\|\hY_{t_{2n+2}}^f- \hY_{t_{2n+2}}^c\|^2 
\lesssim (1-2\gamma h) \|\hY_{t_{2n}}^f- \hY_{t_{2n}}^c\|^2 +\left((2\gamma)^{-1}+2\right) K^2\|\hY_{t_{2n+1}}^f- \hY_{t_{2n}}^f\|^2 h.
\end{eqnarray*}
Then  we multiply by $\e^{\gamma t_{2n+2}}$ on both sides and $\e^{2\gamma h}\lesssim 2$ gives
\begin{eqnarray*}
\e^{\gamma t_{2n+2}}\|\hY_{t_{2n+2}}^f- \hY_{t_{2n+2}}^c\|^2 \!\!\!\!&\lesssim& \!\!\!\! \e^{\gamma t_{2n}} \|\hY_{t_{2n}}^f \!-\! \hY_{t_{2n}}^c\|^2 +(\gamma^{-1}\!+\!4)K^2\,\e^{\gamma t_{2n}}\|\hY_{t_{2n+1}}^f- \hY_{t_{2n}}^f\|^2 h.
\end{eqnarray*}
Summing over multiple timesteps and noting $\hY_{t_0}^f- \hY_{t_0}^c=0 $ gives
\begin{eqnarray*}
\e^{\gamma t_{2n}}\|\hY_{t_{2n}}^f- \hY_{t_{2n}}^c\|^2 &\lesssim &  (\gamma^{-1}\!+\!4)K^2 \sum_{k=0}^{n-1} \e^{\gamma t_{2k}}\|\hY_{t_{2k+1}}^f- \hY_{t_{2k}}^f\|^2 h.
\end{eqnarray*}
Then raising both sides to the power $p/2,$ taking supremum over $n\in[0,N/2],$ taking expectation and by Jensen's inequality, we obtain
\begin{eqnarray*}
\EE\left[ \sup_{0\leq n \leq N/2}\!\e^{\gamma p t_{2n}/2}\|\hY_{t_{2n}}^f\!-\! \hY_{t_{2n}}^c\|^p \right] \!\!\!\!&\lesssim&\!\!\!\!  (\gamma^{-1}\!+\!4)^{p/2} K^p\, \EE\left[ \left| \sum_{k=0}^{N/2-1} \e^{\gamma t_{2k} }\|\hY_{t_{2k+1}}^f\!-\! \hY_{t_{2k}}^f\|^2 h \right|^{p/2}
\right]\\
&&\!\!\!\!\!\!\!\!\!\!\!\!\!\!\!\!\!\!\!\!\!\!\!\!\!\!\!\!\!\!\!\!\!\!\!\!\!\!\!\!\!\!\!\!\!\!\!\!\!\!\!\!\!\!
 \leq (\gamma^{-1}\!+\!4)^{p/2} K^p \left| \sum_{k=0}^{N/2-1}\! \e^{\gamma t_{2k} } h\right|^{p/2-1} \!\!\!\!\! \EE\left[  \sum_{k=0}^{N/2-1}\! \e^{\gamma t_{2k} }\|\hY_{t_{2k+1}}^f- \hY_{t_{2k}}^f\|^p h 
\right]\\
&&\!\!\!\!\!\!\!\!\!\!\!\!\!\!\!\!\!\!\!\!\!\!\!\!\!\!\!\!\!\!\!\!\!\!\!\!\!\!\!\!\!\!\!\!\!\!\!\!\!\!\!\!\!\! \leq    (\gamma^{-1}\!+\!4)^{p/2} K^p \left| \sum_{k=0}^{N/2-1}\!\! \e^{\gamma t_{2k} } h\right|^{p/2-1} \!\! \sum_{k=0}^{N/2-1} \e^{\gamma t_{2k} }\EE\left[\|\hY_{t_{2k+1}}^f- \hY_{t_{2k}}^f\|^p \right]\!h .
\end{eqnarray*}
By the update on fine path, Lipschitz condition (\ref{Lipschitz condition}), Theorem \ref{Theorem: stability of path} and Jensen's inequality, there exists a constant $C_1>0$ such that
\begin{eqnarray}
\label{small increment}
\EE\left[\|\hY_{t_{2k+1}}^f- \hY_{t_{2k}}^f\|^p \right] &=& \EE\left[\|f(\hY_{t_{2k}}^f)h+ S(\hY_{t_{2k}}^c\!-\!\hY_{t_{2k}}^f)h+\Delta W_{2k}\|^{p}\right] \nonumber\\
&\leq & 2^{p-1}\EE\left[ \|f(\hY_{t_{2k}}^f)+ S(\hY_{t_{2k}}^c\!-\!\hY_{t_{2k}}^f)\|^p\right]h^p + 2^{p-1}\EE\left[ \|\Delta W_{2k}\|^p\right]\nonumber\\
&\lesssim&  C_1^p p^{p/2} h^{p/2},
\end{eqnarray}
which implies that
\begin{equation}
\label{half order}
\EE\left[ \sup_{0\leq n \leq N/2}\!\!\!\e^{\gamma p t_{2n}/2}\|\hY_{t_{2n}}^f\!-\! \hY_{t_{2n}}^c\|^p \right] \lesssim 2^{-p/2}  (\gamma^{-1}\!+\!4)^{p/2} K^p C_1^p\, p^{p/2} \e^{\gamma p T/2} h^{p/2} .
\end{equation}

\textbf{Second}, we directly multiply by $\e^{2\gamma t_{2n+2}}$ on both sides of \eqref{error:2} and $\e^{4\gamma h}\lesssim 2$ gives
\begin{eqnarray*}
\e^{2\gamma t_{2n+2}}\|\hY_{t_{2n+2}}^f- \hY_{t_{2n+2}}^c\|^2 &\lesssim&  \e^{2\gamma t_{2n}} \|\hY_{t_{2n}}^f- \hY_{t_{2n}}^c\|^2 +4\,\e^{2\gamma t_{2n}}K^2\|\hY_{t_{2n+1}}^f- \hY_{t_{2n}}^f\|^2 h^2\\
&&\!\!\!\!\!\!\!\!\!\!\!\!\!\!\!\!\!\!\!\!\!\!\!\!\!\!\!\!\!+\,2(1-4Sh+2S^2h^2)\,\e^{2\gamma t_{2n+2}} \langle \hY_{t_{2n}}^f- \hY_{t_{2n}}^c,f(\hY_{t_{2n+1}}^f)- f(\hY_{t_{2n}}^f) \rangle h.
\end{eqnarray*}
Summing over multiple timesteps and noting that $\hY_{t_0}^f- \hY_{t_0}^c=0 $ gives
\begin{eqnarray*}
\e^{2\gamma t_{2n}}\|\hY_{t_{2n}}^f- \hY_{t_{2n}}^c\|^2 &\lesssim &  4\sum_{k=0}^{n-1} \e^{2\gamma t_{2k}}K^2\|\hY_{t_{2k+1}}^f- \hY_{t_{2k}}^f\|^2 h^2\\
&&\!\!\!\!\!\!\!\!\!\!\!\!\!\!\!\!\!\!\!\!\!\!\!\!\!\!\!\!\! +\,2(1-4Sh+2S^2h^2) \sum_{k=0}^{n-1} \e^{2\gamma t_{2k+2}}\langle \hY_{t_{2k}}^f- \hY_{t_{2k}}^c,f(\hY_{t_{2k+1}}^f)- f(\hY_{t_{2k}}^f) \rangle h.
\end{eqnarray*}
Then raising both sides to the power $p/2,$ taking supremum over $n\in[0,N/2],$ taking expectation and by Jensen's inequality, we obtain
\begin{eqnarray*}
\EE\left[ \sup_{0\leq n \leq N/2}\e^{\gamma p t_{2n}}\|\hY_{t_{2n}}^f- \hY_{t_{2n}}^c\|^p \right] &\lesssim&  2^{p/2-1}4^{p/2}(I_1+I_2), 
\end{eqnarray*}
where
\begin{eqnarray*}
I_1 &=& \EE\left[\ \left|\sum_{k=0}^{N/2-1} \e^{2\gamma t_{2k}}K^2\|\hY_{t_{2k+1}}^f- \hY_{t_{2k}}^f\|^2 h^2\right|^{p/2}\ \right],
\\
I_2 &=& \EE\left[\sup_{0\leq n\leq N/2-1}\left|  \sum_{k=0}^{n} \e^{2\gamma t_{2k}}\langle \hY_{t_{2k}}^f- \hY_{t_{2k}}^c,f(\hY_{t_{2k+1}}^f)- f(\hY_{t_{2k}}^f) \rangle h \right|^{p/2}\ \right].
\end{eqnarray*}

For $I_1,$ Jensen's inequality and the estimate \eqref{small increment} give
\begin{eqnarray*}
I_1 &\leq&  \left|\sum_{k=0}^{N/2-1} \e^{2\gamma t_{2k}}h\right|^{p/2-1}\sum_{k=0}^{N/2-1} \e^{2\gamma t_{2k}}h K^p\EE\left[\|\hY_{t_{2k+1}}^f- \hY_{t_{2k}}^f\|^p\right]  h^{p/2}  \\
&\lesssim& (2\gamma)^{-p/2} K^p C_1^p\, p^{p/2}\, \e^{\gamma pT} h^p.
\end{eqnarray*}

For $I_2,$ we perform a Taylor expansion and by mean value theorem obtain
\begin{eqnarray*}
\langle \hY_{t_{2k}}^f- \hY_{t_{2k}}^c,f(\hY_{t_{2k+1}}^f)- f(\hY_{t_{2k}}^f) \rangle &=& \langle \hY_{t_{2k}}^f- \hY_{t_{2k}}^c,\nabla f(\hY_{t_{2k}}^f)(\hY_{t_{2k+1}}^f- \hY_{t_{2k}}^f) \rangle + R_k \\
&&\!\!\!\!\!\!\!\!\!\!\!\!\!\!\!\!\!\!\!\!\!\!\!\!\!\!\!\!\!\!\!\!\!\!\!\!\!\!\!\!\!\!\!\!\!\!\!\!\!\!\!\!\!\!\!\!\!\! =\langle \hY_{t_{2k}}^f- \hY_{t_{2k}}^c,\nabla f(\hY_{t_{2k}}^f)f(\hY_{t_{2k}}^f) \rangle h  +\langle \hY_{t_{2k}}^f- \hY_{t_{2k}}^c,\nabla f(\hY_{t_{2k}}^f)\Delta W_{2k} \rangle \\
&& \!\!\!\!\!\! +S\langle \hY_{t_{2k}}^f- \hY_{t_{2k}}^c,\nabla f(\hY_{t_{2k}}^f)(\hY_{t_{2k}}^c-\hY_{t_{2k}}^f) \rangle h +R_k ,
\end{eqnarray*}
where $|R_k| \leq 2K\|\hY_{t_{2k}}^f- \hY_{t_{2k}}^c\|\|\hY_{t_{2k+1}}^f- \hY_{t_{2k}}^f \|^2$
and then by Jensen's inequality, we have
\begin{eqnarray*}
I_2 & = & 4^{p/2-1}(J_1 + J_2 +J_3+J_4),
\end{eqnarray*}
where
\begin{eqnarray*}
J_1&=&\EE\left[\sup_{0\leq n\leq N/2-1}\left|\sum_{k=0}^{n} \e^{2\gamma t_{2k}}\langle \hY_{t_{2k}}^f- \hY_{t_{2k}}^c,\nabla f(\hY_{t_{2k}}^f)f(\hY_{t_{2k}}^f) \rangle h^2 \right|^{p/2}\right],
\\
J_2 &=&\EE\left[\sup_{0\leq n\leq N/2-1}\left|\sum_{k=0}^{n} \e^{2\gamma t_{2k}} S \langle \hY_{t_{2k}}^f- \hY_{t_{2k}}^c,\nabla f(\hY_{t_{2k}}^f)(\hY_{t_{2k}}^c-\hY_{t_{2k}}^f) \rangle h^2 \right|^{p/2}\right],
\\
J_3 &=&\EE\left[\sup_{0\leq n\leq N/2-1}\left|\sum_{k=0}^{n} \e^{2\gamma t_{2k}}\langle \hY_{t_{2k}}^f- \hY_{t_{2k}}^c,\nabla f(\hY_{t_{2k}}^f) \Delta W_{2k} \rangle h \right|^{p/2}\right],\\
J_4 &=&\EE\left[\sup_{0\leq n\leq N/2-1}\left|\sum_{k=0}^{n} \e^{2\gamma t_{2k}}R_k\, h\ \right|^{p/2}\right].
\end{eqnarray*}
For $J_1,$ by Cauchy-Schwarz inequality, Jensen's inequality, Young's inequality, the Lipschitz property of $f$ and $\nabla f$ and Theorem \ref{Theorem: stability of path}, for any $\zeta>0,$ there exists a constant $C_{31}>0$ such that
\vspace{-0.2em}
\begin{eqnarray*}
J_1 \!\!\!\! &\leq & \!\!\!\!  \EE\left[ \left|\sum_{k=0}^{N/2-1} \e^{2\gamma kh}h\right|^{p/2-1}\!\! \sum_{k=0}^{N/2-1} \e^{2\gamma kh}h \e^{\gamma p t_{2k}/2}\| \hY_{t_{2k}}^f\!\!-\!\! \hY_{t_{2k}}^c\|^{p/2}\|\nabla f(\hY_{t_{2k}}^f)f(\hY_{t_{2k}}^f) \|^{p/2}h^{p/2}\right]\\
&&\!\!\!\!\!\!\!\!\!\!\!\!\!\!\leq  \EE\left[\left|\sup_{0\leq n\leq N/2}\!\!\e^{\gamma p t_{2n}/2}\| \hY_{t_{2n}}^f\!-\! \hY_{t_{2n}}^c\|^{p/2}\right|\! \left|\sum_{k=0}^{N/2-1}\!\! \e^{2\gamma kh}h\right|^{p/2-1}\!\!\! \sum_{k=0}^{N/2-1}\!\! \e^{2\gamma kh}h \|\nabla f(\hY_{t_{2k}}^f)f(\hY_{t_{2k}}^f) \|^{p/2}h^{p/2}\right]\\
&&\!\!\!\!\!\!\!\!\!\!\!\!\!\! \leq\frac{1}{4\zeta}  \EE\left[\sup_{0\leq n\leq N/2}\!\!\e^{\gamma p t_{2n}}\| \hY_{t_{2n}}^f\!-\! \hY_{t_{2n}}^c\|^{p}\right] \!+\! \zeta \EE\left[ \left|\sum_{k=0}^{N/2-1}\!\! \e^{2\gamma kh}h\right|^{p-1}\!\sum_{k=0}^{N/2-1}\!\! \e^{2\gamma kh}h \|\nabla f(\hY_{t_{2k}}^f)f(\hY_{t_{2k}}^f) \|^{p}h^{p}\right]\\
&&
\!\!\!\!\!\!\!\!\!\!\!\!\!\!\leq  \frac{1}{4\zeta}  \EE\left[\sup_{0\leq n\leq N/2}\e^{\gamma p t_{2n}}\| \hY_{t_{2n}}^f- \hY_{t_{2n}}^c\|^{p}\right]\! +\! \zeta \,C_{31}^p\, p^{p}\, \e^{\gamma pT}\, h^p.
\end{eqnarray*}
Similarly, for $J_2,$ there exists a constant $C_{32}>0$ such that, for any $\zeta>0,$
\[
J_2 \leq  \frac{1}{4\zeta}  \EE\left[\sup_{0\leq n\leq N/2}\e^{\gamma p t_{2n}}\| \hY_{t_{2n}}^f- \hY_{t_{2n}}^c\|^{p}\right] + \zeta \,C_{32}^p\, p^{p}\, \e^{\gamma pT}\, h^p.
\]
Then, for $J_3,$ by the Burkholder-Davis-Gundy inequality in \cite{BY82}, the Lipschitz property of $\nabla f$ and Theorem \ref{Theorem: stability of path}, there exists a constant $C_{33}>0$ such that, for any $\zeta>0,$
\begin{eqnarray*}
J_3 &\leq & (C_{\scaleto{\mathrm{BDG}}{4pt}}\,p)^{p/4}\ \EE\left[\left|\sum_{k=0}^{N/2-1} \e^{4\gamma t_{2k}}\| \hY_{t_{2k}}^f- \hY_{t_{2k}}^c\|^2 \|\nabla f(\hY_{t_{2k}}^f)\|^2  h^3 \right|^{p/4}\right]\\
&&\!\!\!\!\!\!\!\!\!\!\!\!\!\!\!\!\!\!\!\!\!\!\!\!\!\!\leq (C_{\scaleto{\mathrm{BDG}}{4pt}}\, p)^{p/4}\EE\left[\left|\sup_{0\leq n\leq N/2}\!\! \e^{\frac{\gamma p t_{2n}}{2}}\| \hY_{t_{2n}}^f\!-\! \hY_{t_{2n}}^c\|^{\frac{p}{2}}\right|  \left|\sum_{k=0}^{N/2-1} \e^{2\gamma t_{2k}}h\right|^{\frac{p}{4}-1}\!\! \sum_{k=0}^{N/2-1}\!\! \e^{2\gamma t_{2k}} h\|\nabla f(\hY_{t_{2k}}^f)\|^{\frac{p}{2}}  h^{\frac{p}{2}} \right]\\
&\leq & \frac{1}{4\zeta}\  \EE\left[\sup_{0\leq n\leq N/2}\e^{\gamma p t_{2n}}\| \hY_{t_{2n}}^f- \hY_{t_{2n}}^c\|^{p}\right] + \zeta\,C_{33}^p\, p^{p}\,\e^{\gamma pT}\, h^p.
\end{eqnarray*}
Similarly, for $J_4,$ by Jensen's inequality, Young inequality, Lipschitz property of $f$ and $\nabla f$ and Theorem \ref{Theorem: stability of path}, for any $\zeta>0,$ there exists a constant $C_{34}>0$ such that, for any $\zeta>0,$
\begin{eqnarray*}
J_4 &\leq& \EE\left[\left|\sum_{k=0}^{N/2-1} \e^{2\gamma t_{2k}}h\right|^{p/2-1} \sum_{k=0}^{N/2-1} \e^{2\gamma t_{2k}}h  \e^{\gamma p t_{2k}/2}\| \hY_{t_{2k}}^f\!\!-\!\! \hY_{t_{2k}}^c\|^{p/2} K^{p/2} \|(\hY_{t_{2k+1}}^f\!\!-\!\! \hY_{t_{2k}}^f)\|^{p}\right]\\
&&\!\!\!\!\!\!\!\!\!\!\!\!\!\!\!\!\!\!\!\!\!\!\!\!\leq  \EE\left[\left|\sup_{0\leq n\leq N/2}\!\! \e^{\gamma p t_{2n}/2}\| \hY_{t_{2n}}^f\!-\! \hY_{t_{2n}}^c\|^{p/2}\right| \left|\sum_{k=0}^{N/2-1} \e^{2\gamma t_{2k}}h\right|^{p/2-1} 
\sum_{k=0}^{N/2-1} \e^{2\gamma t_{2k}}h K^{p/2} \|(\hY_{t_{2k+1}}^f- \hY_{t_{2k}}^f)\|^{p}\right]\\
&&\!\!\!\!\!\!\!\!\!\!\!\!\leq  \frac{1}{4\zeta}\  \EE\left[\sup_{0\leq n\leq N/2}\e^{\gamma p t_{2n}}\| \hY_{t_{2n}}^f- \hY_{t_{2n}}^c\|^{p}\right]\\
&&\!\!\!\!\!\!\!\!\!\!\!\!\!\!\!\!\!\!\!\!\!\!\!\!+ \zeta\, \EE\left[ \left|\sum_{k=0}^{N/2-1} \e^{2\gamma t_{2k}}h\right|^{p-1}\sum_{k=0}^{N/2-1} \e^{2\gamma t_{2k}}h            
K^{p}2^{2p-1} \left(\|f(\hY_{t_{2k}}^f)\!+\! S(\hY_{t_{2k}}^c- \hY_{t_{2k}}^f)\|^{2p}h^{2p}\!+\!\|\Delta W_{2k}\|^{2p}\right)
\right]\\
&&\!\!\!\!\!\!\!\!\!\!\!\!\lesssim  \frac{1}{4\zeta}  \EE\left[\sup_{0\leq n\leq N/2}\e^{\gamma p t_{2n}}\| \hY_{t_{2n}}^f- \hY_{t_{2n}}^c\|^{p}\right] +\zeta \, C_{34}^p\,p^p \e^{\gamma pT}\,  h^p .
\end{eqnarray*}
Finally, by choosing $\zeta=2^{5p/2-2},$ there exists a constant $C_{4}>0$ such that 
\begin{equation*}
\EE\left[\sup_{0\leq n\leq N/2}\e^{\gamma p t_{2n}}\| \hY_{t_{2n}}^f- \hY_{t_{2n}}^c\|^{p}\right] \lesssim C_{4}^p\, p^{p}\, \e^{\gamma pT} h^p,
\end{equation*}
which together with \eqref{half order} implies that there exists a constant $C_5>0$ such that
\begin{equation*}
\sup_{0\leq n\leq N/2}\EE\left[\| \hY_{t_{2n}}^f- \hY_{t_{2n}}^c\|^{p}\right] \lesssim C_{5}^p \, \min\left( p^{p/2}\,  h^{p/2},\ p^{p}\,  h^p \right).
\end{equation*}
For the odd time points, Lipschitz condition \eqref{Lipschitz condition} and  \eqref{eq:one_sided_Lipschitz3} gives
\begin{eqnarray}
\label{error:1}
\|\hY_{t_{2n+1}}^f\!-\!\hY_{t_{2n+1}}^c\|^2 \!\!\!\!&\leq &\!\!\!\! \left( (1-2Sh)^2 + 2h\lambda (1-2Sh) + K^2 h^2 \right) \|\hY_{t_{2n}}^f\!-\! \hY_{t_{2n}}^c\|^2 \nonumber\\
&\lesssim &\!\!\!\! 2\,\|\hY_{t_{2n}}^f- \hY_{t_{2n}}^c\|^2,
\end{eqnarray}
which, by raising both sides to power $p/2$ and taking expectation, there exist constants $C_{(1)},\ C_{(2)}>0$ such that, $\forall\ 0<h<C_{(1)},$
\[
\sup_{0\leq n\leq N}\EE\left[\| \hY_{t_{n}}^f- \hY_{t_{n}}^c\|^{p}\right] \leq C_{(2)}^p\, \min\left( p^{p/2}\,  h^{p/2},\ p^{p}\,  h^p \right).
\]
\end{proof}

\subsection{Theorem \ref{lemma 1}}
\label{proof 4}
\begin{proof}
For simplicity, we only show the proof for $\frac{\D \widehat{\mathbb{Q}}^c}{\D \mathbb{P}}$, and the result for $\frac{\D \mathbb{Q}^f}{\D \mathbb{P}}$ follows similarly. 

Now we write down the detail of the exact Radon-Nikodym derivative given in \eqref{eq: RD}. 
\begin{eqnarray*}
\EE\left[\left| \frac{\D \widehat{\mathbb{Q}}^c}{\D \mathbb{P}} \right|^p \right] \!\!\!\!\!\!&=&\!\!\!\!\!\!\EE\left[ \exp\left(\!\! -pS\!\!\sum_{n=0}^{N/2-1}\!\!\left\langle \hY_{t_{2n}}^f\!-\!\hY_{t_{2n}}^c,\ \Delta W_{2n}\!+\! \Delta W_{2n+1}\right\rangle \!-\!\frac{p}{2}S^2\!\! \sum_{n=0}^{N/2-1}\left\|\hY_{t_{2n}}^f\!-\!\hY_{t_{2n}}^c\right\|^2\!\! 2h\right)\!\!\right]\\
 &=&\EE\left[ \exp\left(\frac{p(2p-1)}{2}S^2 \sum_{n=0}^{N/2-1}\left\|\hY_{t_{2n}}^f-\hY_{t_{2n}}^c\right\|^2 2h\right)\right.\\
&&\!\!\!\!\!\!\!\!\!\!\!\!\!\!\!\!\!\!\!\!\!\!\!\!\!\!\!\!\!\!\!\!\!\!\!\!\!\!\!\!\!\!\!\!\!\!\!\left.\times\exp\left(-pS\sum_{n=0}^{N/2-1}\left\langle \hY_{t_{2n}}^f\!-\!\hY_{t_{2n}}^c,\ \Delta W_{2n}\!+\! \Delta W_{2n+1}\right\rangle-p^2
S^2 \sum_{n=0}^{N/2-1}\left\|\hY_{t_{2n}}^f-\hY_{t_{2n}}^c\right\|^2 2h \right)\right],
\end{eqnarray*}
and then the H{\"o}lder's inequality gives
\[\EE\left[\left| \frac{\D \widehat{\mathbb{Q}}^c}{\D \mathbb{P}} \right|^p \right]\ \leq\ I_1^{1/2}\,I_2^{1/2},\]
where
\begin{eqnarray*}
I_1 \!\!\!\!\!\!& = &\!\!\!\!\!\!\EE\left[\exp\left(p(2p-1)S^2 \sum_{n=0}^{N/2-1}\left\|\hY_{t_{2n}}^f-\hY_{t_{2n}}^c\right\|^2 2h\right)\right],\\
I_2 \!\!\!\!\!\!& = &\!\!\!\!\!\! \EE\left[ \exp\left(-2pS\!\!\sum_{n=0}^{N/2-1}\!\!\left\langle \hY_{t_{2n}n}^f\!-\!\hY_{t_{2n}}^c,\Delta W_{2n}\!+\! \Delta W_{2n+1}\right\rangle\!-\!2p^2
S^2\! \sum_{n=0}^{N/2-1}\!\left\|\hY_{t_{2n}}^f\!-\!\hY_{t_{2n}}^c\right\|^2 \! 2h \right)\right]\!\!\leq \!\!1,\\
\end{eqnarray*} 
since the exponential term in $I_2$ is a super-martingale.

For $I_1,$ by Fatou's lemma and Jensen's inequality, we obtain
\begin{eqnarray*}
I_1\
&\leq& \
\sum_{k=0}^\infty \frac{\EE\left[\left| 2(pS)^2 \sum_{n=0}^{N/2-1} \|\hY_{t_{2n}}^f-\hY_{t_{2n}}^c \|^2 2h\right|^k\right]}{k!}\\
&\leq&\ \sum_{k=0}^\infty \frac{2^k(pS)^{2k} T^{k-1} \sum_{n=0}^{N/2-1}\EE\left[\left\|\hY_{t_{2n}}^f-\hY_{t_{2n}}^c\right\|^{2k}\right] 2h }{k!}, 
\end{eqnarray*}
then by Theorem \ref{thm: strong error} and the Stirling's approximation $k!\geq \sqrt{2\pi}k^{k+1/2}\e^{-k}$ for any $k\geq1$, there exist constants $C_1,\ C_2 >0$ such that 
\begin{eqnarray*}
I_1\
&\lesssim &\  1 + \sum_{k=1}^\infty \frac{(2pSC_2)^{2k} (T h/\e)^k}{\sqrt{2\pi k}} <2.
\label{RNnorms}
\end{eqnarray*}
provided $(2pSC_2)^2 Th/\e<1/2.$ 

Therefore, for all $T>0$ and $p\geq 1,$ there exist constants $C_{(1)},\ C_{(2)}>0$ such that, for all  $0<h<\min(C_{(1)},C_{(2)}/(Tp^2)),$
\[\EE\left[\left| \frac{\D \widehat{\mathbb{Q}}^c}{\D \mathbb{P}} \right|^p \right]\  \leq 2.\]
\end{proof}

\subsection{Theorem \ref{thm: level variance}}
\label{proof 3}
\begin{proof}
By Jensen's and H{\"o}lder's inequalities, we split the expectation into three parts:
\begin{eqnarray*}
&\!\!\!\!\!\!\!\!\!\!\!\!\!\!\!\!\!\!\!\!\!\!\!\!\!\!\!\!\!\!\!\!\!\!\!\!\!\!\!\!\!\!\!\!\!\!\!\!\!\!\!\!\!\!\!\!\!\!\!\!\!\!\!\!\!\!\!\!\!\!\!\!\!\!\!\!\!\!\!\!\!\!\!\!\!\!\!\!\!\!\!\!\!\!\!\!\!\!\!\!\!\!\!\!\!\!\!\!\!\!\!\!\!\!\!\!\!\!\!\!\!\!\!\!\!\!\!\!\!\!\!\!\!\!\!\!\!\!\!\!\EE\left[\left| \varphi(\hY_T^f)\frac{\D  \widehat{\mathbb{Q}}^f}{\D \mathbb{P}}-\varphi(\hY_T^c)\frac{\D  \widehat{\mathbb{Q}}^c}{\D \mathbb{P}}\right|^{p}\right] \\
& =\EE\left[\left| \varphi(\hY_T^f)\frac{\D  \widehat{\mathbb{Q}}^f}{\D \mathbb{P}}-\varphi(\hY_T^f) +\varphi( \hY_T^f) - \varphi(\hY_T^c) + \varphi(\hY_T^c) - \varphi(\hY_T^c)\frac{\D  \widehat{\mathbb{Q}}^c}{\D \mathbb{P}}\right|^{p}\right]
\\
& \ \ \ \ \leq 3^{p-1}\EE\left[\left| \varphi(\hY_T^f)-\varphi(\hY_T^c)\right|^{p}\right] +3^{p-1}\EE\left[\left|\varphi( \hY_T^f) \right|^{2p}\right]^{1/2}\EE\left[\left|1-\frac{\D  \widehat{\mathbb{Q}}^f}{\D \mathbb{P}}\right|^{2p}\right]^{1/2}\\
&\!\!\!\!\!\!\!\!\!\!\!\!\!\!\!\!\!\!\!\!\!\!\!\!\!\!\!\!\!\!\!\!\!\!\!\!\!\!\!\!\!\!\!\!\!\!\!\!\!\!\!\!\!\!\!\!+3^{p-1}\EE\left[\left| \varphi(\hY_T^c) \right|^{2p}\right]^{1/2}\EE\left[\left|1-\frac{\D  \widehat{\mathbb{Q}}^c}{\D \mathbb{P}}\right|^{2p}\right]^{1/2}.
\end{eqnarray*}
By Theorems \ref{Theorem: stability of path} and \ref{thm: strong error} and Lipschitz condition, there exists a constant $C_1>0$ such that 
\begin{eqnarray*}
&\EE\left[\left| \varphi(\hY_T^f)-\varphi(\hY_T^c)\right|^{p}\right] \leq K^p\,\EE \left[\left\| \hY_T^f-\hY_T^c\right\|^{p}\right] \lesssim  C_1^p\, p^{p}\, h^{p},\\
&\!\!\!\EE\left[\left| \varphi(\hY_T^f) \right|^{2p}\right] \lesssim C_1^p\, p^p,\ \ \ \ \ \ \ \
\EE\left[\left| \varphi(\hY_T^c) \right|^{2p}\right] \lesssim C_1^p\, p^p.
\end{eqnarray*}
Next we estimate 
$
\EE\left[\left|1-\frac{\D \widehat{\mathbb{Q}}^c}{\D \mathbb{P}}\right|^{p}\right]$
with
$
\frac{\D \mathbb{Q}^c}{\D \mathbb{P}} = \exp(\mathcal{H})
$
where
\[
\mathcal{H}\triangleq-S\!\sum_{n=0}^{N/2-1}\!\left\langle \hY_{2n}^f\!-\!\hY_{2n}^c,\ \Delta W_{2n}\!+\! \Delta W_{2n+1}\right\rangle-\frac{S^2}{2}\sum_{n=0}^{N/2-1}\left\|\hY_{2n}^f-\hY_{2n}^c\right\|^2 2h.
\]
Taylor expansion gives
\[\e^x=1+\e^{\xi(x)}x,\ \text{for some}\ \xi(x)\ \text{with}\ |\xi(x)|<|x|,\]
which, combined with H{\"o}lder's inequality, implies that
\begin{equation*}
\EE\left[\left|1-\frac{\D \widehat{\mathbb{Q}}^c}{\D \mathbb{P}}\right|^{2p}\right] = \EE\left[\left( \exp(\xi(\mathcal{H}))\,|\mathcal{H}|\right)^{2p}\right]\leq \EE\left[\left| \exp(\xi(\mathcal{H}))\right|^{4p}\right]^{1/2}\EE\left[\left|\mathcal{H}\right|^{4p}\right]^{1/2}.
\end{equation*}
First, by Theorem \ref{lemma 1}, we obtain
\begin{eqnarray*}
\EE\left[\left| \exp(\xi(\mathcal{H}))\right|^{4p}\right]  \leq \EE\left[ \max( \exp(4p\mathcal{H}),1)\right]
\leq  \EE\left[  \exp(4p\mathcal{H})\right] +1 \lesssim 3,
\end{eqnarray*}
provided $h<C_{2}/(T\, p^2)$ for some constant $C_2>0.$

Second, by Jensen's inequality and the Burkholder-Davis-Gundy inequality in \cite{BY82}, there exists a constant $C_3>0$ such that
\begin{eqnarray*}
\EE\left[\left| \mathcal{H}\right|^{4p}\right] 
&\leq &2^{4p-1}S^{4p}\ \EE\left[
\left| \sum_{n=0}^{N/2-1}\left\langle \hY_{2n}^f\!-\!\hY_{2n}^c,\ \Delta W_{2n}\!+\! \Delta W_{2n+1}\right\rangle\right|^{4p}\right]\\
&& +  2^{4p-1}S^{8p}\ \EE\left[\left| \sum_{n=0}^{N/2-1}\left\|\hY_{2n}^f-\hY_{2n}^c\right\|^2 h\right|^{4p}
\right]\\
&& \!\!\!\!\!\!\!\!\!\!\!\!\!\!\!\!\!\!\!\!\!\!\!\!\!\!\!\!\!\!\!\!\!\!\!\!\!\!\!\!\!\!\!\!\!\!\!\!\!\leq2^{4p-1}S^{4p} C_3^p p^{2p}\EE\left[\left| \sum_{n=0}^{N/2-1}\left\|\hY_{2n}^f\!-\!\hY_{2n}^c\right\|^2\! 2h \right|^{2p}\right] \!+\!  2^{4p-1}S^{8p} T^{4p-1}  \EE\left[ \sum_{n=0}^{N/2-1}\left\|\hY_{2n}^f\!-\!\hY_{2n}^c\right\|^{8p}\! h
\right] \\
&& \!\!\!\!\!\!\!\!\!\!\!\!\!\!\!\!\!\!\!\!\!\!\!\!\!\!\!\!\!\!\!\!\!\!\!\!\!\!\!\!\!\!\!\!\!\!\!\!\!\leq 2^{6p-1}S^{4p} C_3^pp^{2p}T^{2p-1} \EE\left[ \sum_{n=0}^{N/2-1}\left\|\hY_{2n}^f\!-\!\hY_{2n}^c\right\|^{4p}\!h\right]\! +\! 2^{4p-1}S^{8p}T^{4p-1} \EE\left[ \sum_{n=0}^{N/2-1}\left\|\hY_{2n}^f\!-\!\hY_{2n}^c\right\|^{8p}\! h
\right].
\end{eqnarray*}
Then, provided $h<C_{4}/\sqrt{T p}$ for some constant $C_4>0,$ by Theorem \ref{thm: strong error}, there exists a constant $C_5>0$ such that
\begin{equation*}
\EE\left[\left| \mathcal{H}\right|^{4p}\right] \lesssim  C_5^{4p}\,p^{6p}\,T^{2p}\,h^{4p},
\end{equation*}
and then
\[\EE\left[\left|1-\frac{\D \widehat{\mathbb{Q}}^c}{\D \mathbb{P}}\right|^{2p}\right]\lesssim  3^{1/2}\, C_5^{2p}\,p^{3p}\,T^{p}\,h^{2p}. \]
Similarly, we can get the same result for $\EE\left[\left|1-\frac{\D \widehat{\mathbb{Q}}^f}{\D \mathbb{P}}\right|^{2p}\right].$

Since $1/\sqrt{Tp}\geq 1/(\sqrt{T}p)\geq \min(1,1/(Tp^2)),$ the condition $h<C_4/\sqrt{Tp}$ can be replaced by $h\leq C_4\min(1,1/(Tp^2)).$ Overall, combining all the estimates above, there exist constants $C_{(1)},$ $C_{(2)},$ $C_{(3)}>0$ such that, for any $0< h< \min(C_{(1)},C_{(2)}/(Tp^2)),$
\[
\EE\left[\left| \varphi(\hY_T^f)\frac{\D \widehat{\mathbb{Q}}^f}{\D \mathbb{P}}-\varphi(\hY_T^c)\frac{\D \widehat{\mathbb{Q}}^c}{\D \mathbb{P}}\right|^{p}\right] \leq \,C_{(3)}^p\, p^{2p}\, T^{p/2}\, h^p.
\]

\end{proof}

\section{Conclusions and future work}
In this paper, we introduced a change of measure technique for multilevel Monte Carlo estimators. For chaotic ergodic SDEs satisfying the one-sided Lipschitz condition, we reduce the exponential increase of the variance to a linear increase, which greatly reduces the computational cost when our interest is in the expectation with respect to the invariant measure. The numerical results support our analysis.

One direction for extension of the theory is to perform the numerical analysis for ergodic SDEs with a non-globally Lipschitz drift using adaptive timesteps and the change of measure technique, since numerical experiments in section \ref{Non-Lipschitz section} show it works well in these cases. Another possible direction is to follow the idea in \cite{GB17} to estimate the mean exit time and related path functionals which correspond to the solution of an elliptic PDEs. In this case, we need to estimate the mean exit time in the infinite time interval and the associated path functionals, which again is a long-time simulation problem.

\section*{References}

\bibliography{citation}



\end{document}